\documentclass{amsart}
\usepackage[T1]{fontenc}
\usepackage{amsmath}
\usepackage{amssymb}
\usepackage{amsthm}
\usepackage{mathrsfs}
\usepackage{enumitem}
\usepackage{xcolor}
\usepackage{url}
\usepackage[colorlinks=true, linkcolor=black, citecolor=black, urlcolor=black]{hyperref}
\usepackage[margin=1in]{geometry}
\usepackage{breqn}
\usepackage{mathtools}
\usepackage{adjustbox}
\usepackage[normalem]{ulem}
\usepackage{tabularx}
\usepackage{makecell}
\usepackage{tablefootnote}
\usepackage{comment}
\usepackage{tikz-cd} 
\tikzcdset{  cells={font=\everymath\expandafter{\the\everymath\displaystyle}},
} 
\usepackage{subcaption}
\usepackage{graphicx}

\newcommand{\Z}{\mathbb{Z}}
\newcommand{\Q}{\mathbb{Q}}
\newcommand{\F}{\mathbb{F}}
\renewcommand{\P}{\mathbb{P}}

\newcommand{\Y}{\mathcal{Y}}
\newcommand{\X}{\mathcal{X}}
\renewcommand{\O}{\mathcal{O}}
\newcommand{\A}{\mathbb{A}}
\newcommand{\G}{\mathbb{G}}
\renewcommand{\setminus}{\smallsetminus}
\newcommand{\cmod}[1]{\ (#1)}

\DeclareMathOperator{\lcm}{lcm}
\DeclareMathOperator{\Proj}{Proj}
\DeclareMathOperator{\Ht}{ht}
\DeclareMathOperator{\Spec}{Spec}

\newcommand{\p}{\mathfrak{p}}

\newcommand{\Op}{\O_{K,\p}}

\newcommand{\Frobp}{\mathrm{Frob}_{\p}}
\newcommand{\ProbK}{\Prob^{(K)}}
\newcommand{\Np}{N_K(\p)}

\newcommand{\Prob}{\rho}
\newcommand{\Probvals}[3]{\Prob_{#1,#2,#3}}
\newcommand{\Probzero}{\Prob_0}
\newcommand{\ProbA}{\Prob_A}
\newcommand{\ProbB}{\Prob_B}
\newcommand{\ProbC}{\Prob_C}
\newcommand{\ProbAB}{\Prob_{AB}}
\newcommand{\ProbAC}{\Prob_{AC}}
\newcommand{\ProbBC}{\Prob_{BC}}
\newcommand{\ProbBCnotpower}{\sigma}
\newcommand{\ProbvalsBCnotpower}[3]{\ProbBCnotpower_{#1,#2,#3}}

\newcommand{\glm}{g_{\ell m}}
\newcommand{\gln}{g_{\ell n}}
\newcommand{\gmn}{g_{mn}}

\newcommand{\bv}{\boldsymbol{v}}

\newcommand{\bu}{\boldsymbol{u}}

\numberwithin{equation}{section}
\numberwithin{table}{section}

\newtheorem{theorem}[equation]{Theorem}
\newtheorem{proposition}[equation]{Proposition}

\newtheorem{lemma}[equation]{Lemma}
\newtheorem{corollary}[equation]{Corollary}

\theoremstyle{definition}
\newtheorem{definition}[equation]{Definition}
\newtheorem{example}[equation]{Example}

\newtheorem{remark}[equation]{Remark}

\newcommand{\ZZ}{\mathbb{Z}}
\newcommand{\Gm}{\mathbb{G}_{m}}
\newcommand{\PP}{\mathbb{P}}
\newcommand{\PPP}{\mathcal{P}}
\renewcommand{\AA}{\mathbb{A}}
\newcommand{\OO}{\mathcal{O}}
\DeclareMathOperator{\sep}{sep}
\DeclareMathOperator{\Stab}{Stab}

\newcommand{\Rng}{R}
\newcommand{\x}{x}
\newcommand{\y}{y}
\newcommand{\z}{z}
\renewcommand{\l}{\ell}
\newcommand{\m}{m}
\newcommand{\n}{n}

\renewcommand{\m}{m}

\newcommand{\bfw}{\mathbf{w}}
\newcommand{\gfe}{A\x^{\l} + B\y^{\m} + C\z^{\n}}
\newcommand{\RR}{\mathcal{R}}
\newcommand{\UU}{\mathcal{U}}

\DeclarePairedDelimiter\abs{\lvert}{\rvert}

\newcommand{\ex}{\ell}
\newcommand{\ey}{m}
\newcommand{\ez}{n}
\newcommand{\Xcrs}{X}

\newcommand{\ptX}{P}
\newcommand{\ptY}{Q}
\newcommand{\ptcrs}[1]{\abs{#1}}

\newcommand{\canonicaldiv}{K}
\renewcommand{\L}{L}
\newcommand{\f}{f}
\renewcommand{\k}{k}
\newcommand{\K}{K} 

\newcommand{\pcond}[1]
{#1}
\newcommand{\expcond}[1]
{#1}

\title{On $p$-adic solubility of $Ax^\ell + By^m + Cz^n = 0$}

\author[C. Keyes]{Christopher Keyes}
\address{Christopher Keyes, Center for Communications Research, Princeton, NJ, USA}
\email{ckeyes.math@gmail.com}
\urladdr{\url{https://c-keyes.github.io}}

\author[A. Kobin]{Andrew Kobin}
\address{Andrew Kobin, Center for Communications Research, La Jolla, CA, USA}
\email{ajkobinmath@gmail.com}
\urladdr{\url{https://www.andrewkobin.com}}

\makeatletter
\let\@wraptoccontribs\wraptoccontribs
\makeatother

\contrib[with an appendix by]{Santiago Arango-Pi\~{N}Eros, Christopher Keyes and Andrew Kobin}
\address{Santiago Arango-Pi\~neros, Department of Mathematics, University of Massachusetts Amherst, Amherst, MA, USA}
\email{santiago.arango.pineros@gmail.com}
\urladdr{\url{https://sarangop1728.github.io}}



\begin{document}

\begin{abstract}
    We study $p$-adic solubility of generalized Fermat equations $Ax^\ell + By^m + Cz^n = 0$ for positive integers $\ell,m,n$. For all but finitely many primes $p$, the probability of having a $p$-adic solution is described by a rational function in $p$ depending only on $\gcd(p-1,\ell,m)$, $\gcd(p-1,\ell,n)$, and $\gcd(p-1,m,n)$. When $\ell,m,n$ are pairwise coprime, we deduce that the proportion of these equations which are everywhere locally soluble is positive, given by a product of these local probabilities; when $\ell,m,n$ are not pairwise coprime, the proportion is 0\%. We then give several detailed examples demonstrating the explicit nature of the results.
\end{abstract}

\maketitle

\section{Introduction}
\label{sec:intro}

Fix $\ell,m,n \geq 1$. For nonzero integers $A,B,C$, the Diophantine equation
\begin{equation}\label{eq:gfe}
	Ax^\ell + By^m + Cz^n = 0
\end{equation}
is known as a \textit{generalized Fermat equation}. Of particular interest are the \textit{primitive} integer solutions to \eqref{eq:gfe}, i.e.\ $(x,y,z)$ with $xyz \neq 0$ and $\gcd(x,y,z) = 1$. A necessary condition for \eqref{eq:gfe} to have any primitive integer solutions is for it to have primitive $p$-adic solutions for all primes $p$. That is, for all $p$ there must exist $(x,y,z) \in \Z_p^3 \setminus (p\Z_p)^3$ satisfying \eqref{eq:gfe}.

Our first main result is that for all but finitely many primes $p$, the probability $\Prob_{\ell,m,n}(p)$ of \eqref{eq:gfe} having a primitive $\Z_p$-solution is computed by one of finitely many rational functions in $p$ (see Definition \ref{def:rho} for a precise definition). Moreover, this rational function expression depends only on $\gcd(p-1,\ell,m)$, $\gcd(p-1,\ell,n)$, and $\gcd(p-1,m,n)$. 

\begin{theorem}[see Theorem \ref{thm:rat_funcs_body}]\label{thm:rat_funcs}
    Fix $\ell,m,n \geq 1$ and set $\glm = \gcd(\ell,m)$, $\gln = \gcd(\ell,n)$, and $\gmn = \gcd(m,n)$. Suppose we have divisors $i \mid g_{\ell m}$, $j \mid g_{\ell n}$, and $k \mid g_{mn}$. There exists a rational function $R_{i,j,k}(t) \in \Q(t)$ such that for all primes $p$ satisfying \pcond{$p \nmid \glm\gln\gmn$} and 
    \[\pcond{p+1-\left (1 + \frac{1}{2}\left(\frac{\ell mn}{\lcm(\ell,m,n)} - (\glm + \gln + \gmn)\right)\right)\lfloor{2 \sqrt{p}} \rfloor - \max(i,j,k) > 0}, \]
    we have $\Prob_{\ell,m,n}(p) = R_{i,j,k}(p)$ whenever $i = \gcd(p-1,g_{\ell m}),\ j = \gcd(p-1,g_{\ell n}),\ k = \gcd(p-1,g_{mn})$.
\end{theorem}

The proof of Theorem \ref{thm:rat_funcs} gives an effective algorithm to compute $\Prob_{\ell,m,n}(p)$ explicitly for all but finitely many primes $p$;  implementation is available on GitHub \cite{github_gfedensity}. The remaining primes can often be handled via an ad hoc approach. 

\begin{example}[see \S\ref{subsec:332}]
Suppose $(\ell,m,n) = (3,3,2)$. For $p > 3$, we have $(i,j,k) = (3,1,1)$ or $(1,1,1)$, corresponding to $p \equiv 1 \pmod{3}$ and $p \equiv 2 \pmod{3}$, respectively. Following the proof of Theorem \ref{thm:rat_funcs_body}, we compute 
\begin{align*}
    R_{3,1,1}(t) &= 1 - \frac{2  t^{11} + 4  t^{10} + 14  t^{9} + 10  t^{8} + 14  t^{7} + 8  t^{6} + 4  t^{5} + 4  t^{4} - 8  t^{2} - 6  t - 6}{3  {\left(t^{5} + t^{4} + t^{3} + t^{2} + t + 1\right)} {\left(t^{2} + t + 1\right)}^{2} {\left(t^{2} + 1\right)} t}, \\
    R_{1,1,1}(t) &= 1 - \frac{4  t^{9} + 2  t^{8} + 4  t^{7} + 2  t^{6} + 2  t^{4} - 2  t^{2} - 2  t - 2}{{\left(t^{5} + t^{4} + t^{3} + t^{2} + t + 1\right)} {\left(t^{2} + t + 1\right)}^{2} {\left(t^{2} + 1\right)} t}.
\end{align*}
Then Theorem~\ref{thm:rat_funcs_body} shows that $\Prob_{3,3,2}(p) = R_{i,j,k}(p)$ for all \pcond{$p > 3$}. In \S \ref{subsubsec:332_p=3}, we also show that $\Prob_{3,3,2}(2) = R_{1,1,1}(2)$ and adapt the methods appropriately for $p=3$. Thus we conclude
\begin{equation*}
    \Prob_{3,3,2}(p) = \begin{cases}
        1 - \frac{2  p^{11} + 4  p^{10} + 14  p^{9} + 10  p^{8} + 14  p^{7} + 8  p^{6} + 4  p^{5} + 4  p^{4} - 8  p^{2} - 6  p - 6}
        {3  {\left(p^{5} + p^{4} + p^{3} + p^{2} + p + 1\right)} {\left(p^{2} + p + 1\right)}^{2} {\left(p^{2} + 1\right)} p}
        & p \equiv 1 \pmod{3}, \\
        1 - \frac{4  p^{9} + 2  p^{8} + 4  p^{7} + 2  p^{6} + 2  p^{4} - 2  p^{2} - 2  p - 2}
        {{\left(p^{5} + p^{4} + p^{3} + p^{2} + p + 1\right)} {\left(p^{2} + p + 1\right)}^{2} {\left(p^{2} + 1\right)} p}
        & p \equiv 2 \pmod{3},\\
        \frac{424871}{461370} & p=3.
    \end{cases}
\end{equation*}
\end{example}

\begin{remark}
    For fixed $\ell,m,n$ and a prime $p$, the corresponding triple $(i,j,k)$ is completely determined by how $p$ splits in the cyclotomic extension $\Q(\zeta)$, where $\zeta$ is a primitive $\lcm(\glm,\gln,\gmn)$-th root of unity.
    
	More generally, in \S\ref{subsec:NFs} we discuss an analogue of Theorem~\ref{thm:rat_funcs} that holds for a prime ideal $\p$ in an arbitrary number field $K$. For all but finitely many $\p$, the probability of \eqref{eq:gfe} having a primitive $\p$-adic integral solution is given by a rational function expression in the norm $N_K(\p)$, whose coefficients depend on the splitting type of $\p$ in the cyclotomic extension $K(\zeta)/K$. See Theorem \ref{thm:rat_funcs_NFs} and the preceding discussion for details.
    
	Much of this article can be read with a number field $K$ in place of $\Q$ and a prime ideal $\p$ in place of a rational prime $p$. For simplicity and accessibility, we choose to develop everything over $\Q$ until \S\ref{subsec:NFs}. 
\end{remark}

The special case of $(\ell,m,n) = (n,n,n)$ was addressed by Hirakawa and Kanamura, who describe more generally how to compute the probability of local solubility for diagonal hypersurfaces of degree $n$ in projective space \cite{HirakawaKanamura}. The explicit probabilities produced by Theorem \ref{thm:rat_funcs} recover their results \cite[Theorems 1.3(1) and 1.4(1)]{HirakawaKanamura}; see \S \ref{subsec:nnn} for a discussion.

The special case of $(\ell,m,n) = (2,2,n)$ for $n$ odd was also considered in a recent preprint due to Duque-Rosero, Roy, Sankar, Wang, and the authors, where a local solubility probability was obtained for \eqref{eq:gfe} with $A=1$ and $p \neq 2$ \cite[Remark 7.2]{DRKKRSW_22n}. In \S \ref{subsec:22n}, we give closed form expressions for $\Prob_{2,2,n}(p)$ for $p \neq 2$.

Theorem \ref{thm:rat_funcs} and the explicit probabilities it produces fit into a wider landscape of studying solubility in families, many of which arise from fibrations. A modern framework for these types of questions has been developed by several authors; see e.g.\ \cite{LoughranSmeets, BBL,LoughranSofos,LoughranRomeSofos2022,LoughranMatthiesen}. 

Fix exponents $(\ell,m,n)$. We say \eqref{eq:gfe} is \textit{everywhere locally soluble} if it has real solutions and primitive $\Z_p$-solutions for all primes $p$. For a rational point in the projective plane $P \in \P^2(\Q)$, write $P = [A:B:C]$ for integers $A,B,C$ with $\gcd(A,B,C)=1$ and let $\Ht(P) = \max\{|A|,|B|,|C|\}$ denote the usual height function on $\P^2(\Q)$. Note that the solubility of \eqref{eq:gfe} is invariant under scaling $A,B,$ and $C$, so it is well defined for any representative of $P \in \P^2(\Q)$.

For any real number $T > 0$, let
\begin{align*}
    N_{\ell,m,n}(T) &= \#\left\{ P \in \P^2(\Q) : \Ht(P) \leq T,\ \eqref{eq:gfe} \text{ has a primitive } \Z\text{-solution} \right\},\\
    N_{\ell,m,n}^{\mathrm{loc}}(T) &= \#\left\{ P \in \P^2(\Q) : \Ht(P) \leq T,\ \eqref{eq:gfe} \text{ is everywhere locally soluble} \right\}.
\end{align*}
We prove that the natural density of $(A,B,C)$ for which \eqref{eq:gfe} is everywhere locally soluble is positive when the exponents $\ell,m,n$ are pairwise coprime, and is 0 otherwise.

\begin{theorem}\label{thm:stats}
    The asymptotic growth of $N_{\ell,m,n}^{\mathrm{loc}}(T)$ can be described as follows.
    \begin{enumerate}[label = (\roman*)]
        \item If \expcond{$\ell,m,n$ are pairwise coprime},
        \[\frac{N_{\ell,m,n}^{\mathrm{loc}}(T)}{\#\{P \in \P^2(\Q) : \Ht(P) \leq T\}} \sim \prod_p \Prob_{\ell,m,n}(p) > 0.\]
        \item If instead \expcond{$\ell,m,n$ are not pairwise coprime},
        \[\limsup_{T \to \infty} \frac{N_{\ell,m,n}^{\mathrm{loc}}(T)}{\#\{P \in \P^2(\Q) : \Ht(P) \leq T\}} = 0.\]
    \end{enumerate}
\end{theorem}

As with Theorem~\ref{thm:rat_funcs}, these statements admit generalizations to number fields; see \S\ref{subsec:NFs}. 

To date, most of the families considered in the literature arise from fibrations of varieties over affine or projective space, but one may also pose various solubility questions for fibrations of stacks, as in \cite{DRKKRSW_22n}. Theorem~\ref{thm:stats} can be interpreted as a partial answer to the local solubility question for fibrations of stacky curves determined by generalized Fermat equations; see \S\ref{subsec:GFEs} and \cite[\S 2]{DRKKRSW_22n} for more details on this geometric interpretation. 

\begin{example}[see \S \ref{subsec:numerics}]
    Suppose $(\ell,m,n) = (2,3,5)$, which are pairwise coprime. Then we have
    \[0.78233 \leq \lim_{T \to \infty} \frac{N_{2,3,5}^{\mathrm{loc}}(T)}{\#\{P \in \P^2(\Q) : \Ht(P) \leq T\}} \leq 0.78237.\]
    That is, the probability (in the sense of density of $[A:B:C] \in \P^2(\Q)$) of \eqref{eq:gfe} being everywhere locally soluble is about 78.2\%. This is deduced from Theorem \ref{thm:stats} and an explicit description of $\Prob_{2,3,5}(p)$ obtained by Theorem \ref{thm:rat_funcs_body}. 
    
    Darmon and Granville \cite[p.\ 540]{DarmonGranville} speculate on whether everywhere local solubility suffices to guarantee existence of a primitive integral solution\footnote{That is, they ask whether the Hasse principle for integral points holds for \eqref{eq:gfe} when $(\ell,m,n) = (2,3,5)$.} in this case. A resolution to their question could yield a precise asymptotic formula for $N_{2,3,5}(T)$.
\end{example}

The proof of Theorem~\ref{thm:stats}(i) uses the sieve of Ekedahl, essentially following the argument in \cite[Theorem 1.3]{BBL} with minor modification. In the special case of $(\ell,m,n) = (n,n,n)$, much more is known than Theorem~\ref{thm:stats}(ii). Browning and Dietmann gave asymptotic upper bounds for $N_{n,n,n}^{\mathrm{loc}}(T)$ \cite{BrowningDietmann} which were recently supplanted by Koymans, Paterson, Santens, and Schute \cite{KPSS_locsolfermat}, who determine an asymptotic
\[N_{n,n,n}^{\mathrm{loc}}(T) \sim c_n T^3 (\log T)^{3\alpha_n - 3}\]
for explicitly described constants $\alpha_n < 1$ and $c_n > 0$. 

For the case when $(\ell,m,n) = (2,2,n)$, for $n$ odd, and $A = 1$ in \eqref{eq:gfe}, Duque-Rosero, Roy, Sankar, Wang, and the authors show that the number of $B,C \in (\Z \cap [-T,T])^2$ for which \eqref{eq:gfe} is everywhere locally soluble is bounded above and below by a constant times $T^2/\sqrt{\log T}$ \cite[Theorem C]{DRKKRSW_22n}. In forthcoming work, the same authors study local solubility for a broader family of stacks given by fibrations.

Density formulae and asymptotic results for local solubility similar to Theorems \ref{thm:rat_funcs} and \ref{thm:stats} for other families of interest appear in the literature: everywhere local solubility for quadric hypersurfaces is studied in \cite{BhargavaCremonaFisherJonesKeating}, for plane cubics in \cite{BhargavaCremonaFisher_plane_cubics}, for genus $1$ curves in $\P^{1}\times\P^{1}$ given by a form of bidegree $(2,2)$ in \cite{FisherHoPark}, for univariate polynomial equations of arbitrary degree in \cite{BhargavaCremonaFisherGajovic}, for genus $4$ trigonal superelliptic curves in \cite{BeneishKeyes2023}, and for cubic hypersurfaces in \cite{BeneishKeyes2025}. 

Some asymptotic results for \textit{global} solubility are also known, for example for quadrics, \cite{Guo,FriedlanderIwaniec,BhargavaCremonaFisherJonesKeating}, hyperelliptic curves \cite{bhargava2013hyperellipticcurvesqrational}, cubic surfaces \cite{Browning_cubic_surfaces}, hypersurfaces \cite{poonenvoloch,BLeBS,BeneishKeyes2025}, Ch\^{a}telet varieties \cite{BST2022, Diao2025}, and spherical stacky curves \cite{DRKKRSW_22n}. 
For an overview of this area of research, including many more references, see the introduction to \cite{LoughranRomeSofos2022}. 

\subsection{Organization}

In \S \ref{sec:prelim}, we give a geometric description of primitive solutions to generalized Fermat equations and some preliminary results for detecting $p$-adic solutions. In \S \ref{sec:probs} we define the local probabilities $\Prob_{\ell,m,n}(p)$ and describe them in terms of several conditional probabilities, which we then relate to one another in \S \ref{sec:relations}. The proofs of Theorems \ref{thm:rat_funcs} and \ref{thm:stats} are given in \S \ref{sec:proofs}. Several explicit examples are discussed in \S \ref{sec:examples}, with the relevant code hosted on GitHub \cite{github_gfedensity}; explicit formulae for $\Prob_{\ell,m,n}(p)$ are given for selected $(\ell,m,n)$ in \S \ref{sec:table}. Finally, in Appendix \ref{sec:appendix} (joint with Santiago Arango-Pi\~{n}eros), we describe the geometry of generalized Fermat equations in the language of stacky curves, deriving explicit Euler characteristic formulas and illustrating the general theory with several examples. 

\subsection*{Acknowledgments}

The authors would like to thank Santiago Arango-Pi\~{n}eros, Tim Browning and Dan Loughran for helpful conversations and feedback on an early draft of the main article. 

CK was partially supported by the Additional Funding Programme for Mathematical Sciences, delivered by EPSRC (EP/V521917/1) and the Heilbronn Institute for Mathematical Research, as well as an AMS-Simons Travel Grant.

\section{Preliminaries}
\label{sec:prelim}

\subsection{Notation}

Let $\ell,m,n$ denote positive integers throughout. We set the following notation.
\begin{align*}
	\glm &= \gcd(\ell,m) & \gln &= \gcd(\ell,n) & \gmn &= \gcd(m,n) \\
	w_0 &= \frac{\lcm(\ell,m,n)}{\ell} & w_1 &= \frac{\lcm(\ell,m,n)}{m} & w_\infty &= \frac{\lcm(\ell,m,n)}{n}
\end{align*}

\begin{definition}[Primitive]
\label{def:primitive}
    Let $R$ be a commutative ring with unity. A triple $(x,y,z) \in R^3$ is said to be \textit{primitive} if $x,y,z$ generate the unit ideal.
\end{definition}

When $R$ is a principal ideal domain, primitivity is equivalent to $\gcd(x,y,z) = 1$. We will primarily be interested in primitive integer triples $(x,y,z) \in \Z^3$ and primitive $p$-adic triples $(x,y,z) \in \Z_p^3$.

\subsection{Solutions to generalized Fermat equations}
\label{subsec:GFEs}

For a full discussion of generalized Fermat equations, especially the behavior of their primitive $\Z$-solutions, see \cite{DarmonGranville,Beukers,PSS,wilcox-grechuk24}. 

Following \cite{PSS,SantiThesis}, primitive solutions to \eqref{eq:gfe} can be studied geometrically as points on the quasi-affine surface
\begin{equation}\label{eq:SABC}
S_{A,B,C} \coloneqq \Spec\Z[x,y,z]/(Ax^{\ell} + By^{m} + Cz^{n})\setminus\{(0,0,0)\} \subset \A_{\Z}^{3}. 
\end{equation}
Alternatively, one may identify ``equivalent'' solutions by forming the quotient stack 
\begin{equation}\label{eq:XABC}
    \X_{A,B,C} \coloneqq \left[S_{A,B,C}/\G_{m}\right] \subset \left[(\A_{\Z}^{3}\setminus\{(0,0,0)\})/\G_{m}\right] = \mathcal{P}(w_0,w_1,w_\infty),
\end{equation}
where $\G_{m}$ acts on $\A_{\Z}^{3} \setminus \{(0,0,0)\}$ with weights $w_0 = \frac{\lcm(\ell,m,n)}{\ell}$, $w_1 = \frac{\lcm(\ell,m,n)}{m}$, and $w_\infty = \frac{\lcm(\ell,m,n)}{n}$. 

Geometrically, $\X_{A,B,C}$ is a relative stacky curve in the weighted projective stack $\mathcal{P}(w_0, w_1, w_\infty)$. Its coarse moduli space, denoted $X_{A,B,C}$, is given by $\Proj \left(\Spec \Z[x,y,z]/(Ax^\ell + By^m + Cz^n)\right)$ where $x,y,z$ have degrees $w_0, w_1, w_\infty$, respectively. The scheme $X_{A,B,C}$ is a relative curve with fibers of genus 
    \begin{equation}
    \label{eq:genus_coarse}
        g(X_{A,B,C}) = 1 + \frac12\left( \frac{\ell m n }{\lcm(\ell,m,n)} - (\glm + \gln + \gmn)\right)
    \end{equation}
(see Corollary~\ref{cor:GFE-euler-char}). Moreover, after base change to $R=\Z[\frac1{ABC\ell m n}]$, the coarse moduli map $(\X_{A,B,C})_R \to (X_{A,B,C})_R$ restricts to an isomorphism away from $xyz=0$ \cite[Lemma 3.2.2.b (iii)-(iv)]{SantiThesis}. The geometry of these curves is described further in the Appendix. 

\begin{lemma}
For any integers $\ell,m,n \geq 1$ and $A,B,C \neq 0$, and any prime $p$, 
\begin{enumerate}[label = (\roman*)]
    \item Equation \eqref{eq:gfe} has a primitive $\Z$-solution if and only if $\X_{A,B,C}(\Z)\not = \emptyset$; 
    \item Equation \eqref{eq:gfe} has a primitive $\Z_{p}$-solution if and only if $\X_{A,B,C}(\Z_{p})\not = \emptyset$. 
\end{enumerate}
\end{lemma}

\begin{proof}
The quotient map $S\rightarrow\X$ induces essential surjections of groupoids $S(\Z)\to\X(\Z)$ and $S(\Z_{p})\to\X(\Z_{p})$ and in particular these are surjective on objects. Both statements follow immediately. 
\end{proof}

In most situations, solutions to \eqref{eq:gfe} are described equally well by $S_{A,B,C}$ and $\X_{A,B,C}$. For the purposes of detecting $\Z$- or $\Z_p$-points, however, the relative dimension $1$ object $\X_{A,B,C}$ is more natural to work with. Additionally, while the coarse space $X_{A,B,C}$ is sufficient for detecting \emph{rational} solutions to \eqref{eq:gfe}, due to the failure of the valuative criterion for properness for algebraic stacks, analyzing $X_{A,B,C}$ is insufficient for detecting integral (or $p$-adic integral) solutions; see e.g.\ \cite{BhargavaPoonen,santens23,DRKKRSW_22n}. 

For a prime $p$, let $\X_{A,B,C}(\Z_{p})$ (resp.~$X_{A,B,C}(\Z_{p})$) denote the set\footnote{This is an abuse of notation: for a stack $\X$, $\X(R)$ typically denotes the {\it groupoid} of $R$-points of $\X$, but we will only need to consider objects (= equivalence classes of solutions to \eqref{eq:gfe}) and so $\X_{A,B,C}(R)$ will just denote a set for us.} of $\Z_{p}$-points on $\X_{A,B,C}$ (resp.~$X_{A,B,C}$). Readers unfamiliar with the language of algebraic stacks may read the notation $\X_{A,B,C}(\Z_{p})$ as ``the set of $p$-adic integer solutions to the generalized Fermat equation \eqref{eq:gfe} up to scaling'' (with weights $w_0,w_1,w_\infty$) and not lose any insight into our results.

Furthermore, we will use $\X_{A,B,C}(\Z_p; *)$ to denote the subset of points satisfying some condition $*$, e.g.\ 
\[\X_{A,B,C}(\Z_p; p \nmid x) = \left\{[x:y:z] \in \X_{A,B,C}(\Z_p) : p \nmid x \right\}.\]
Here $[x:y:z]$ denotes the equivalence class of triples $(x,y,z) \in \Z_p^3$ under the \textit{weighted} $\G_m$-action identifying $(x,y,z)$ with $(\lambda^{w_0}x, \lambda^{w_1}y, \lambda^{w_\infty}z)$ for $\lambda \in \Z_p^\times$.
We will always use $*$ for divisibility conditions on the coordinates; such conditions are well defined on equivalence classes.

\subsection{\texorpdfstring{$p$-adic points on $\X_{A,B,C}$}{p-adic points on X_A,B,C}}
\label{subsec:padic_pts}

In this subsection, we assemble some useful technical lemmas for detecting $p$-adic solutions to \eqref{eq:gfe}, viewed as $\Z_{p}$-points on the stacky curve $\X_{A,B,C}$. 

\begin{lemma}\label{lem:intermediate}
	Suppose $\expcond{\gmn=1}$ and $p \nmid B,C$. Then there exist $y,z \in \Z_p^\times$ satisfying
	\[By^m + Cz^n = 0.\]
	In particular, for any $A$ we have $\X_{A,B,C}(\Z_p; p \nmid y)$ and $\X_{A,B,C}(\Z_p; p \nmid z)$ are both nonempty.
\end{lemma}

\begin{proof}
	Write $\alpha m + \beta n = 1$ for $\alpha,\beta \in \Z$. Without loss of generality, assume $\beta n$ is odd. Then $y = (C/B)^\alpha$ and $z = (-B/C)^\beta$ satisfies the desired equality. For any value of $A$, we have
	\[\left[0:(C/B)^\alpha:(-B/C)^\beta\right] \in \X_{A,B,C}(\Z_p; p \nmid yz)\]
	so we are done.
\end{proof}

\begin{lemma}\label{lem:p_divides_none}
    Suppose $\pcond{p \nmid \glm\gln\gmn}$ and $p$ satisfies
    \begin{equation}\label{eq:pt_count_hyp}
        \pcond{p+1 - \left(1 + \frac{1}{2}\left (\frac{\ell mn}{\lcm(\ell,m,n)} - (g_{\ell m} + g_{\ell n} + g_{mn})\right )\right)\lfloor 2\sqrt{p} \rfloor - \gcd(p-1,\gmn) > 0.}
    \end{equation}
    If also $p \nmid A,B,C$, then $\X_{A,B,C}(\Z_p; p \nmid x) \neq \emptyset$.
\end{lemma}

\begin{proof}
    Over $\F_p$, the reduction of the coarse space $\overline{X}_{A,B,C}$ is isomorphic to a smooth curve of genus given by \eqref{eq:genus_coarse}.\footnote{If $w_0$, $w_1$, $w_\infty$ are not pairwise coprime, then a priori $X_{A,B,C}$ is not \textit{well-formed} in the sense of \cite{Iano-Fletcher_2000}. However, via an isomorphism of the underlying weighted projective space, we find $X_{A,B,C}$ is isomorphic to a well-formed curve and the genus formula is unchanged. Smoothness follows from $p \nmid ABC\glm\gln\gmn$.} Moreover, $X_{A,B,C}(\F_p; x=0)$ is in bijection with nontrivial solutions to $By^{\gmn} + Cz^{\gmn} = 0$, considered up to scaling. 
    
    There are at most $\gcd(p-1,\gmn)$ such points, so by the Hasse--Weil bound, when $p$ satisfies \eqref{eq:pt_count_hyp} we have $\#X_{A,B,C}(\F_p; x\neq 0) > 0$. Since such a point has at least two nonzero coordinates, and $p$ divides at most one of $\ell,m,n$, it may be lifted via Hensel's lemma to a $\Z_p$-solution to \eqref{eq:gfe}, i.e. a point in $\X_{A,B,C}(\Z_p; p \nmid x)$.
\end{proof}

\begin{remark}\label{rem:improvements_to_pt_ct_bound}
    We can weaken the hypothesis \eqref{eq:pt_count_hyp} in certain situations. For example, if $p > 2$, $p \nmid m$, and $\gcd(p-1,m) = 1$, then \eqref{eq:pt_count_hyp} is not necessary, since $(\F_p^\times)^m = \F_p^\times$. In this case, $X_{A,B,C}(\F_p; x \neq 0)$ always contains a point with nonzero $y$-coordinate, which may be lifted via Hensel's Lemma to $X_{A,B,C}(\Z_p; p \nmid x)$.
\end{remark}

\begin{remark}\label{rem:direct_proof}
    Interpreting solutions to \eqref{eq:gfe} as points on the stacky curve $\X_{A,B,C}$, or its coarse space, is not essential when proving a statement like Lemma \ref{lem:p_divides_none}. One could instead view solutions as points on the smooth quasi-affine surface $S\subset\mathbb{A}^3$ defined in \eqref{eq:SABC} and count $\F_p$-points away from $x=0$ (or $y=0$, $z=0$, respectively), using an effective version of the Lang--Weil theorem (see e.g.\ \cite{Slavov_2023}) in place of \eqref{eq:pt_count_hyp}. 
\end{remark}

\begin{lemma}\label{lem:p_divides_one}
    Suppose $p \mid A$ and $p \nmid BC$. Assume $\expcond{m \leq n}$.
    \begin{enumerate}[label = (\roman*)]
        \item If $\pcond{p \nmid \gmn}$, then $\X_{A,B,C}(\Z_p; p \nmid y) = \X_{A,B,C}(\Z_p; p \nmid z)$ are nonempty if and only if $-C/B \in (\F_p^\times)^{\gmn}$.

        \item If $p^m \nmid A$ then $\X_{A,B,C}(\Z_p; p \mid y,z) = \emptyset$.

        \item If $p^m \mid A$ then $\X_{A,B,C}(\Z_p; p \mid y,z)$ is nonempty if and only if $\X_{p^{-m}A,B,p^{n-m}C}(\Z_p; p \nmid x)$ is nonempty.
    \end{enumerate}
\end{lemma}

\begin{proof}
    Any solution to \eqref{eq:gfe} with $p \nmid y$ must have $p \nmid z$ and vice versa, establishing the equality in (i). A necessary condition for such a solution to exist is for 
    \begin{equation}\label{eq:BCcong}
    	By^m + Cz^n \equiv 0 \pmod{p}
    \end{equation} 
    to have a nontrivial solution in $\F_p$. Moreover, since $\pcond{p \nmid \gmn}$, we have $\pcond{p \nmid m}$ or $\pcond{p \nmid n}$, so any nontrivial $\F_p$-solution to \eqref{eq:BCcong} lifts via Hensel's lemma to a $\Z_p$-solution to \eqref{eq:gfe} with $p \nmid y,z$.

    Thus it suffices to determine when \eqref{eq:BCcong} has a nontrivial $\F_p$-solution. Rewriting \eqref{eq:BCcong} as $(y^{m/\gmn})^{\gmn} \equiv \frac{-C}{B} (z^{n/\gmn})^{\gmn} \pmod{p}$, we see that $-C/B \in (\F_p^\times)^{\gmn}$ is necessary. To see it suffices, suppose $-C/B = D^{\gmn}$ and observe
    \[y^{m/\gmn} \equiv D z^{n/\gmn} \pmod{p} \implies By^m + Cz^n \equiv 0 \pmod{p}.\]
    A solution to the former with $p \nmid y,z$ always exists since $\gcd(m/\gmn, n/\gmn) = 1$, by the argument in the proof of Lemma \ref{lem:intermediate}. This establishes (i).

    For (ii) and (iii), we see that $[x:y:z] \in \X_{A,B,C}(\Z_p; p \mid y,z)$ satisfies $p^m \mid Ax^\ell$. If $p^m \nmid A$, this forces $p \mid x$, a contradiction. If $p^m \mid A$, then $[x:y/p:z/p] \in \X_{p^{-m}A,B,p^{n-m}C}(\Z_p; p \nmid x)$. The reverse implication is similar.
\end{proof}

\begin{lemma}\label{lem:p_divides_two}
    Suppose $p \mid A,B$ and $p \nmid C$. Set $i = \min\{v_p(A), v_p(B), n\}$. 
    \begin{enumerate}[label = (\roman*)]
        \item $\X_{A,B,C}(\Z_p) = \X_{A,B,C}(\Z_p; p \mid z)$.
        \item $\X_{A,B,C}(\Z_p; p \mid z)$ is nonempty if and only if $\X_{p^{-i}A, p^{-i}B, p^{n-i}C}(\Z_p : p \nmid x \text{ or } p \nmid y)$ is nonempty.
    \end{enumerate}
\end{lemma}

\begin{proof}
    The hypotheses force any $[x:y:z] \in \X_{A,B,C}(\Z_p)$ to have $p \mid z$, giving (i). For (ii), we see that if $[x:y:z] \in \X_{A,B,C}(\Z_p; p \mid z)$, then $[x:y:z/p] \in \X_{p^{-i}A, p^{-i}B, p^{n-i}C}(\Z_p; p \nmid x \text{ or } p \nmid y)$, and vice-versa.
\end{proof}

\section{Probabilities}
\label{sec:probs}

For a prime $p$, let $\mu_p$ denote the natural Haar measure on $\Z_p^3$, normalized so $\mu_p(\Z_p^3) = 1$.

\begin{definition}\label{def:rho}
	Let $\Prob_{\ell,m,n}(p)$ denote the probability that $\X_{A,B,C}$ has a $\Z_p$-point,
    \[\Prob_{\ell,m,n}(p) = \mu_p\left(\left\{ (A,B,C) \in \Z_p^3 : \X_{A,B,C}(\Z_p) \neq \emptyset \right\}\right).\]
	For brevity, we will often drop the subscripts and $p$-dependence and write $\Prob = \Prob_{\ell,m,n}(p)$. 
\end{definition}

We also define several conditional probabilities.

\begin{definition}\label{def:rhocond}
    Let
    \begin{align*}
        \Probzero &= \frac{\mu_p\left(\left\{ (A,B,C) \in (\Z_p^\times)^3 : \X_{A,B,C}(\Z_p) \neq \emptyset\right\}\right)}{(p-1)^3/p^3},\\
        \ProbA &= \frac{\mu_p\left(\left\{ (A,B,C) \in p\Z_p \times (\Z_p^\times)^2 : \X_{A,B,C}(\Z_p) \neq \emptyset\right\}\right)}{(p-1)^2/p^3},\\
        \ProbAB &= \frac{\mu_p\left(\left\{ (A,B,C) \in (p\Z_p)^2 \times \Z_p^\times : \X_{A,B,C}(\Z_p) \neq \emptyset\right\}\right)}{(p-1)/p^3}
    \end{align*}
    denote the probabilities that $\X_{A,B,C}$ has a $\Z_p$-point given $p \nmid ABC$, $p$ divides $A$ only, and $p$ divides $A$ and $B$ only, respectively.\footnote{The denominator of $\Probzero$ is $\mu_p((\Z_p^\times)^3)$, the probability of $p \nmid ABC$. Similarly, the denominators of $\ProbA$ and $\ProbAB$ are, respectively, $\mu_p(p\Z_p \times (\Z_p^\times)^2)$ and $\mu_p((p\Z_p)^2 \times \Z_p^\times)$.}
    We define probabilities $\ProbB$, $\ProbC$, $\ProbAC$, and $\ProbBC$ similarly, conditional on $p$ dividing the coefficient(s) in the subscript, but not the others.
\end{definition}

Combining the definitions above, we have
\begin{equation}\label{eq:rho}
	\Prob = \frac{(p-1)^3}{p^3} \Probzero + \frac{(p-1)^2}{p^3} \left( \ProbA + \ProbB + \ProbC \right) + \frac{p-1}{p^3} \left( \ProbAB + \ProbAC + \ProbBC \right) + \frac1{p^3} \Prob.
\end{equation}

\begin{remark}\label{rem:different_def_rho}
    In our definition of $\Prob$, we allow for $(A,B,C) \in (p\Z_p)^3$ to align with the local factors that occur in \cite[Proposition 3.4]{BBL}, which we will use in the proof of Theorem \ref{thm:stats}. It would also be natural to define $\Prob$ by first removing common factors of $p$: 
    \[\widetilde{\Prob}_{\ell,m,n}(p) = 
        \frac{\mu_p\left(\left\{ (A,B,C) \in \Z_p^3 \setminus (p\Z_p)^3 : \X_{A,B,C}(\Z_p) \neq \emptyset \right\}\right)}
        {\mu_p\left(\Z_p^3 \setminus (p\Z_p)^3\right)}.\]
    This essentially amounts to using a Haar measure $\widetilde{\mu}_p$, renormalized so that $\widetilde{\mu}_p\left(\Z_p^3 \setminus (p\Z_p)^3\right) = 1$. After redefining the conditional probabilities in Definition \ref{def:rhocond} with respect to this normalization and forming a relation analogous to \eqref{eq:rho}, it follows from $\mu_p\left(\Z_p^3 \setminus (p\Z_p)^3\right) = 1-1/p^3$ that $\Prob_{\ell,m,n}(p) = \widetilde{\Prob}_{\ell,m,n}(p)$.
\end{remark}

We will also need some auxiliary probabilities that incorporate conditions on the valuations of the coefficients and divisibility restrictions on the coordinates of $[x:y:z] \in \X_{A,B,C}(\Z_p)$.

\begin{definition}\label{def:aux}
    Let $\Probvals{*_A}{*_B}{*_C}$ be the probability that \eqref{eq:gfe} has a $\Z_p$-solution given that $A,B,C$ satisfy the prescribed conditions on their $p$-adic valuations. For example, for $a,b,c \geq 0$, $\Probvals{=a}{=b}{=c}$ denotes the probability of a $\Z_{p}$-solution given that $v_{p}(A) = a,v_{p}(B) = b$ and $v_{p}(C) = c$: 
    \[\Probvals{=a}{=b}{=c} = \frac{\mu_p\left(\left\{ (A,B,C) \in \Z_p^3 : (v_p(A),v_p(B),v_p(C)) = (a,b,c),\ \X_{A,B,C}(\Z_p) \neq \emptyset \right\}\right)}{\mu_p\left(\left\{ (A,B,C) \in \Z_p^3 : (v_p(A),v_p(B),v_p(C)) = (a,b,c)\right\}\right)}.\]
    We may also prescribe an inequality, as in
    \[\Probvals{\geq a}{=b}{=c} = \frac{\mu_p\left(\left\{ (A,B,C) \in \Z_p^3 : v_p(A) \geq a,\ (v_p(B),v_p(C)) = (b,c),\ \X_{A,B,C}(\Z_p) \neq \emptyset \right\}\right)}{\mu_p\left(\left\{ (A,B,C) \in \Z_p^3 : v_p(A) \geq a,\ (v_p(B),v_p(C)) = (b,c)\right\}\right)}.\]
    Furthermore, let $\Probvals{*_A}{*_B}{*_C}^{(x)}$ denote the probability that $\X_{A,B,C}(\Z_p; p \nmid x) \neq \emptyset$, given the prescribed conditions on $A,B,C$ (and similarly for $\Probvals{*_A}{*_B}{*_C}^{(y)}$, $\Probvals{*_A}{*_B}{*_C}^{(z)}$) and $\Probvals{*_A}{*_B}{*_C}^{(x \text{ or } y)}$ denote the conditional probability that $\X_{A,B,C}(\Z_p; p \nmid x \text{ or } p \nmid y) \neq \emptyset$ (and similarly for $\Probvals{*_A}{*_B}{*_C}^{(x \text{ or } z)}$, $\Probvals{*_A}{*_B}{*_C}^{(y \text{ or } z)}$).
\end{definition}

These auxiliary probabilities sometimes overlap with those introduced in Definition \ref{def:rhocond}; for example, $\Probzero = \Probvals{=0}{=0}{=0}$ and $\ProbA = \Probvals{\geq 1}{=0}{=0}$.

Several of these probabilities are straightforward to compute using the results from the previous section.

\begin{lemma}\label{lem:probzero}
	Suppose $\pcond{p \nmid \glm\gln\gmn}$ and \pcond{$p$ satisfies \eqref{eq:pt_count_hyp}}. Then 
	\[\Probzero = \Probvals{=0}{=0}{=0}^{(x)} = 1.\]
\end{lemma}

\begin{proof}
	By Lemma \ref{lem:p_divides_none}, whenever such $p \nmid ABC$ we have $\X_{A,B,C}(\Z_p ; p \nmid x) \neq \emptyset$.
\end{proof}

\begin{lemma}\label{lem:probAyz}
	Suppose $\pcond{p \nmid \gmn}$. For all $a \geq 1$ we have
    \[\Probvals{=a}{=0}{=0}^{(y)} = \Probvals{\geq 1}{=0}{=0}^{(y)} = \Probvals{=a}{=0}{=0}^{(z)} =  \Probvals{\geq 1}{=0}{=0}^{(z)} =  \frac1{\gcd(p-1,g_{mn})}.\]
\end{lemma}

\begin{proof}
    By Lemma \ref{lem:p_divides_one}(i), it suffices to compute the density of $B,C \in \Z_p^\times$ for which $-C/B \in (\F_p^\times)^{\gmn}$. Since for each fixed $B$, $-C/B$ runs through every element of $\F_p^\times$, this is equal to $\#(\F_p^\times)^{\gmn} / \#\F_p^\times = \frac1{\gcd(p-1,g_{mn})}$.
\end{proof}

\begin{remark}\label{rem:probBxz}
    After reordering the variables, Lemma \ref{lem:probzero} can be used to compute $\Probvals{=0}{=0}{=0}^{(y)} = \Probvals{=0}{=0}{=0}^{(z)} =1$ for $p$ sufficiently large, with $\gcd(p-1,\gmn)$ in \eqref{eq:pt_count_hyp} replaced by $\gcd(p-1,\gln)$ and $\gcd(p-1,\glm)$, respectively. 
    Similarly, Lemma \ref{lem:probAyz} can be used to compute 
\begin{align*}
   \Probvals{=0}{=b}{=0}^{(x)} = \Probvals{=0}{\geq 1}{=0}^{(x)} = \Probvals{=0}{=b}{=0}^{(z)} = \Probvals{=0}{\geq 1}{=0}^{(z)} &=  \frac1{\gcd(p-1,g_{\ell n})} \text{ for all  } b \geq 1 \text{ if } \pcond{p \nmid \gln},\\
   \Probvals{=0}{=0}{=c}^{(x)} = \Probvals{=0}{=0}{\geq 1}^{(x)} = \Probvals{=0}{=0}{=c}^{(y)} = \Probvals{=0}{=0}{\geq 1}^{(y)} &=  \frac1{\gcd(p-1,g_{\ell m})}  \text{ for all  } c \geq 1 \text{ if } \pcond{p \nmid \glm}.
\end{align*}
\end{remark}

As a warmup, we can compute $\ProbC$ directly in the special case of $\ell = m$ with $\gmn = 1$. We will tackle the general case in the next section (see Lemma \ref{lem:probA}).

\begin{lemma}\label{lem:probC_l=m}
    Suppose \expcond{$\ell = m$, $\gmn=1$}, and \pcond{$p \nmid m$}. Then 
    \begin{align*}
        \Probvals{=0}{=0}{=c} = \Probvals{=0}{=0}{=c}^{(z)} &= \begin{cases}
            \frac1{\gcd(p-1,m)} & 1 \leq c < m,\\
            1 & c=m,\\
            \Probvals{=0}{=0}{=c-m}^{(z)} & c > m.
        \end{cases} 
    \end{align*}
    It follows that
    \[\ProbC = \Probvals{=0}{=0}{\geq 1}^{(z)} = \left(1 - \frac1{p^m}\right)^{-1}\left(\left(1 - \frac1{p^{m-1}}\right)\frac1{\gcd(p-1,m)} + \frac{p-1}{p^m} \right).\]
\end{lemma}

\begin{proof}
    Suppose $1 \leq c < m$. Then by Lemma \ref{lem:p_divides_one}(ii) (with the coordinates permuted), for any $A,B,C \in \Z_p$ satisfying $(v_p(A),v_p(B),v_p(C)) = (0,0,c)$, we have $\X_{A,B,C}(\Z_p) = \X_{A,B,C}(\Z_p; p \nmid x)$. By Lemma \ref{lem:p_divides_one}(i), it suffices to compute the density of $A,B \in \Z_p^\times$ with $-B/A \in (\F_p^\times)^m$, which is $\frac1{\gcd(p-1,m)}$ by Lemma \ref{lem:probAyz}. Moreover, any nontrivial solution to $Ax^m + By^m \equiv 0 \pmod{p}$ can be lifted to $[x:y:z] \in \X_{A,B,C}$ with $p \nmid z$, so we have $\Probvals{=0}{=0}{=c} = \Probvals{=0}{=0}{=c}^{(z)}$ in this case.

    If $(v_p(A),v_p(B),v_p(C)) = (0,0,m)$, then $\X_{A,B,C}(\Z_p; p \mid y,z)$ is in bijection with $\X_{A,B,p^{-m}C}(\Z_p; p \nmid z)$. However, we can apply Lemma \ref{lem:intermediate} to see the latter is nonempty. Thus $\Probvals{=0}{=0}{=m} = \Probvals{=0}{=0}{=m}^{(z)} = 1$.

    Finally, we claim if $(v_p(A),v_p(B),v_p(C)) = (0,0,c)$ for $c > m$, then $\X_{A,B,C}(\Z_p; p \nmid z) \neq \emptyset$ if and only if $\X_{A,B,p^{-m}C}(\Z_p; p \nmid z) \neq \emptyset$. This is because both have points with $p \nmid x,y$ if and only if the congruence $Ax^m + By^m \equiv 0 \pmod{p}$ has a solution. For points with $p \mid y,z$, we use the bijection in Lemma \ref{lem:p_divides_one}(iii).

    For $C \in p\Z_p$, the probability that $v_p(C)=i > 0$ is $\frac{p-1}{p^i}$. The probability $v_p(C) > i$ is $\frac1{p^i}$. Thus we have    
    \begin{align*}
        \ProbC &= \sum_{i\geq 1} \frac{p-1}{p^i} \Probvals{=0}{=0}{=i}\\
        &= \sum_{1 \leq i < m} \frac{p-1}{p^i} \Probvals{=0}{=0}{=i} + \frac{p-1}{p^m}\Probvals{=0}{=0}{=m} + \frac1{p^m} \Probvals{=0}{=0}{\geq m+1}\\
        &=  \left(1 - \frac1{p^{m-1}}\right)\frac1{\gcd(p-1,m)} + \frac{p-1}{p^m} + \frac1{p^m} \Probvals{=0}{=0}{\geq 1}^{(z)}.
    \end{align*}
    Repeating this calculation for $\Probvals{=0}{=0}{\geq 1}^{(z)}$, we find $\ProbC = \Probvals{=0}{=0}{\geq 1}^{(z)}$. Rearranging gives the result.
\end{proof}

\section{Relations between probabilities}
\label{sec:relations}

If we can determine $\ProbA$ and $\ProbAB$, then we can find $\ProbB$, $\ProbC$, $\ProbAC$, and $\ProbBC$ by permuting the variables. In this section we establish formulae for $\ProbA$ and $\ProbAB$ in terms of the auxiliary probabilities from Definition \ref{def:aux}. We will also need an additional auxiliary probability with an extra condition.

\begin{definition}\label{def:aux_notpower}
    For $a,b,c \geq 1$, let $\ProbvalsBCnotpower{=a}{=b}{=c}^{(x)}$ be the probability that $\X_{A,B,C}(\Z_p; p \nmid x) \neq \emptyset$
   	given that 
    \[(A,B,C) = \left(p^a A_0, p^b B_0, p^c C_0\right) \text{ for } A_0, B_0, C_0 \in \Z_p^\times\]
    and 
    \[-C_0/B_0 \notin (\F_p^\times)^{\gmn}.\]
    That is,
    \[\ProbvalsBCnotpower{=a}{=b}{=c}^{(x)} = \frac{\mu_p \left(\left\{ (A,B,C) \in \Z_p^3 : \substack{(v_p(A), v_p(B),v_p(C)) = (a,b,c),\\ -C_0/B_0 \notin (\F_p^\times)^{\gmn},\\ \X_{A,B,C}(\Z_p; p \nmid x) \neq \emptyset}  \right\} \right)}{\mu_p \left(\left\{ (A,B,C) \in \Z_p^3 : \substack{(v_p(A), v_p(B),v_p(C))=(a,b,c),\\ -C_0/B_0 \notin (\F_p^\times)^{\gmn}}  \right\} \right)}.\]
    Similarly, let
    \[\ProbvalsBCnotpower{\geq 1}{=b}{=c}^{(x)} = \frac{\mu_p \left(\left\{ (A,B,C) \in \Z_p^3 : \substack{p \mid A,\ (v_p(B),v_p(C))=(b,c),\\ -C_0/B_0 \notin (\F_p^\times)^{\gmn},\\ \X_{A,B,C}(\Z_p; p \nmid x) \neq \emptyset}  \right\} \right)}{\mu_p \left(\left\{ (A,B,C) \in \Z_p^3 : \substack{p \mid A,\ (v_p(B),v_p(C))=(b,c),\\ -C_0/B_0 \notin (\F_p^\times)^{\gmn}}  \right\} \right)}.\]
\end{definition}

\begin{lemma}\label{lem:ProbBCnotpower_easy}
    Suppose \pcond{$p \nmid \glm \gln$} and \pcond{$\Probvals{=0}{=0}{=0}^{(x)} = 1$}. Then we have
    \begin{align*}
        \ProbvalsBCnotpower{=0}{=0}{=0}^{(x)} &= 1,\\
        \ProbvalsBCnotpower{=a}{=0}{=0}^{(x)} &= 0 & \text{for } 1 \leq a < \min\{m,n\},\\        
        \ProbvalsBCnotpower{=0}{=b}{=0}^{(x)} &= \frac1{\gcd(p-1,g_{\ell n})} & \text{for } b \geq 1,\\
        \ProbvalsBCnotpower{=0}{=0}{=c}^{(x)} &= \frac1{\gcd(p-1,g_{\ell m})} & \text{for } c \geq 1.
    \end{align*}
\end{lemma}

\begin{proof}
    The first statement follows from the hypothesis $\Probvals{=0}{=0}{=0}^{(x)} = 1$. The second follows from Lemma \ref{lem:p_divides_one}, since $-C/B \notin (\F_p^\times)^{\gmn}$ as part of the condition defining $\ProbvalsBCnotpower{*_A}{*_B}{*_C}$.
    
    For the third, we follow a similar argument as in Lemma \ref{lem:probAyz}: we need to compute the probability that $A,C \in \Z_p^\times$ satisfy $-C/A \in (\F_p^\times)^{\gln}$, but now subject to the condition that $-C/B_0 \notin (\F_p^\times)^{\glm}$. However, these conditions are independent from one another, so we get $\frac1{\gcd(p-1,g_{\ell n})}$ as in Lemma \ref{lem:probAyz} (with the variables appropriately permuted). The last statement follows after swapping the roles of $y$ and $z$.
\end{proof}

We now establish relations between $\ProbA$, $\ProbvalsBCnotpower{=0}{=0}{n-m}^{(x)}$, and $\ProbvalsBCnotpower{\geq 1}{=0}{=n-m}^{(x)}$, and between $\ProbAB$, $\Probvals{\geq 1}{=0}{=c}^{(x)}$, $\Probvals{=0}{\geq 1}{=c}^{(y)}$, $\Probvals{=0}{=0}{=0}^{(x \text{ or } y)}$,  $\Probvals{\geq 1}{=0}{=0}^{(x)}$, and $\Probvals{=0}{\geq 1}{=0}^{(y)}$.

\begin{lemma}\label{lem:probA}
	Suppose $\pcond{p \nmid \glm\gln\gmn}$ and assume $\expcond{m \leq n}$. We have
	\[
	\Probvals{=a}{=0}{=0} = \Probvals{=a}{=0}{=0}^{(x)} = 
		\begin{cases}
			\frac1{\gcd(p-1,g_{mn})} & 1 \leq a < m, \\
			\frac1{\gcd(p-1,g_{mn})} + \left(1 - \frac1{\gcd(p-1,g_{mn})}\right)\ProbvalsBCnotpower{=a-m}{=0}{=n-m}^{(x)} & a \geq m.
		\end{cases}
	\]
	It follows that 
	\begin{align*}
		\ProbA = \Probvals{\geq 1}{=0}{=0}^{(x)} &= \frac{1}{\gcd(p-1,g_{mn})} + \left(1 - \frac{1}{\gcd(p-1,g_{mn})} \right)\left(\frac{p-1}{p^m}\ProbvalsBCnotpower{=0}{=0}{=n-m}^{(x)} + \frac1{p^m}  \ProbvalsBCnotpower{\geq 1}{=0}{=n-m}^{(x)} \right),\\
        \ProbvalsBCnotpower{\geq 1}{=0}{=0}^{(x)} &= \frac{p-1}{p^m} \ProbvalsBCnotpower{=0}{=0}{=n-m}^{(x)} + \frac1{p^m} \ProbvalsBCnotpower{\geq 1}{=0}{=n-m}^{(x)}.
	\end{align*}
\end{lemma}

\begin{proof}
    We have by Lemma \ref{lem:p_divides_one}(i) and the definition of $\ProbvalsBCnotpower{=a}{=0}{=0}^{(x)}$
    \[\Probvals{=a}{=0}{=0} = \frac1{\gcd(p-1,g_{mn})} + \left(1 - \frac1{\gcd(p-1,g_{mn})}\right)\ProbvalsBCnotpower{=a}{=0}{=0}^{(x)}.\]
    If $1 \leq a < m$ then Lemma \ref{lem:ProbBCnotpower_easy} gives $\ProbvalsBCnotpower{=a}{=0}{=0}^{(x)} = 0$. If $a \geq m$, then Lemma \ref{lem:p_divides_one}(iii) gives $\ProbvalsBCnotpower{=a}{=0}{=0}^{(x)} = \ProbvalsBCnotpower{=a-m}{=0}{=n-m}^{(x)}$. The same argument gives the same formulae for $\Probvals{=a}{=0}{=0}^{(x)}$.

    For the second statement, we use that the probability of $v_p(A) = a$ is $\frac{p-1}{p^a}$ to write
    \begin{align*}
        \ProbA &= \sum_{a \geq 1} \frac{p-1}{p^a} \Probvals{=a}{=0}{=0}\\
        &= \frac{1}{\gcd(p-1,g_{mn})} + \left(1 - \frac1{\gcd(p-1,g_{mn})}\right) \sum_{a \geq 1} \frac{p-1}{p^a} \ProbvalsBCnotpower{=a}{=0}{=0}^{(x)} \\
        &= \frac{1}{\gcd(p-1,g_{mn})} + \left(1 - \frac{1}{\gcd(p-1,g_{mn})} \right)\left(\frac{p-1}{p^m}\ProbvalsBCnotpower{=0}{=0}{=n-m}^{(x)} + \frac1{p^m}  \ProbvalsBCnotpower{\geq 1}{=0}{=n-m}^{(x)} \right).
    \end{align*}
    The same calculation gives the same result for $\Probvals{\geq 1}{=0}{=0}^{(x)}$.

    Similarly,
    \begin{align*}
        \ProbvalsBCnotpower{\geq 1}{=0}{=0}^{(x)} &= \sum_{a \geq 1} \frac{p-1}{p^a} \ProbvalsBCnotpower{=a}{=0}{=0}^{(x)} =  \frac{p-1}{p^m} \ProbvalsBCnotpower{=0}{=0}{=n-m} + \frac1{p^m} \ProbvalsBCnotpower{\geq 1}{=0}{=n-m}^{(x)},
    \end{align*}
    giving the last statement.
\end{proof}

\begin{lemma}\label{lem:probAB}
    We have
	\begin{align*}
		\ProbAB = \sum_{1 \leq i < n} & \left( \frac{(p-1)^2}{p^{2i}} \Probvals{=0}{=0}{=n-i}^{(x)} + \frac{p-1}{p^{2i}} \left( \Probvals{\geq 1}{=0}{=n-i}^{(x)} + \Probvals{=0}{\geq 1}{=n-i}^{(y)} \right) \right) \\
		&+ \frac{(p-1)^2}{p^{2n}} \Probvals{=0}{=0}{=0} + \frac{p-1}{p^{2n}}\left( \Probvals{\geq 1}{=0}{=0}^{(x)} + \Probvals{=0}{\geq 1}{=0}^{(y)} \right) + \frac1{p^{2n}}\ProbAB.
	\end{align*}
\end{lemma}

\begin{proof}
    For $a,b \geq 1$ and $i = \min\{a,b,n\}$, Lemma \ref{lem:p_divides_two} gives that
    \[\Probvals{=a}{=b}{0} = \Probvals{=a-i}{=b-i}{=n-i}^{(x \text{ or } y)} = \begin{cases}
        \Probvals{=0}{=0}{=n-a}^{(x)} = \Probvals{=0}{=0}{=n-a}^{(y)} & \text{if } \min\{a,b,n\} = a = b < n,\\[.5em]
        \Probvals{=0}{=b-a}{=n-a}^{(y)} & \text{if } \min\{a,b,n\} = a < b,\\[.5em]
        \Probvals{=a-b}{=0}{=n-b}^{(x)} & \text{if } \min\{a,b,n\} = b < a,\\[.5em]
        \Probvals{=a-n}{=b-n}{=0}^{(x \text{ or } y)} & \text{if } \min\{a,b,n\} = n.
    \end{cases}\]
    We also observe $\Probvals{=0}{=0}{=0}^{(x \text{ or } y)} = \Probzero$. Taking
    \begin{align*}
        \ProbAB &= \sum_{1 \leq i < n} \left( \frac{(p-1)^2}{p^{2i}} \Probvals{=i}{=i}{=0} + \sum_{a > i} \frac{(p-1)^2}{p^{i+a}} \Probvals{=a}{=i}{=0} + \sum_{b > i} \frac{(p-1)^2}{p^{i+b}} \Probvals{=i}{=b}{=0} \right) + \frac1{p^{2n}} \Probvals{\geq n+1}{\geq n+1}{=0},
    \end{align*}
    applying the relations above, and rearranging we arrive at the desired formula. 
\end{proof}

Finally, we establish relations between auxiliary probabilities of the form $\Probvals{\geq 1}{=0}{=c}^{(x)}$ and $\ProbvalsBCnotpower{\geq 1}{=0}{=c}^{(x)}$ for $0 < c < n$, and $\Probvals{\geq 1}{=b}{=0}^{(x)}$ and $\ProbvalsBCnotpower{\geq 1}{=b}{=0}^{(x)}$ for $0 < b < m$. It suffices to do so for the former, from which we can deduce formulae for the latter by swapping $y$ and $z$. The proof is another application of Lemma \ref{lem:p_divides_two}.

\begin{lemma}\label{lem:auxiliary}
	For $0 < c < n$ we have
	\begin{align*}
		\Probvals{\geq 1}{=0}{=c}^{(x)} &= \begin{cases}
			\frac{p-1}{p^c} \Probvals{=0}{=m-c}{=0}^{(x)} + \frac1{p^c} \Probvals{\geq 1}{=m-c}{=0}^{(x)} & \text{if } c < m\\[.5em]
			\frac{p-1}{p^m} \Probvals{=0}{=0}{=0}^{(x)} + \frac1{p^m} \Probvals{\geq 1}{=0}{=0}^{(x)} & \text{if } c = m\\[.5em]
			\frac{p-1}{p^m} \Probvals{=0}{=0}{=c-m}^{(x)} + \frac1{p^m} \Probvals{\geq 1}{=0}{=c-m}^{(x)} & \text{if } c > m
		\end{cases}\\
		\ProbvalsBCnotpower{\geq 1}{=0}{=c}^{(x)} &= \begin{cases}
			\frac{p-1}{p^c} \ProbvalsBCnotpower{=0}{=m-c}{=0}^{(x)} + \frac1{p^c} \ProbvalsBCnotpower{\geq 1}{=m-c}{=0}^{(x)} & \text{if } c < m\\[.5em]
			\frac{p-1}{p^m} \ProbvalsBCnotpower{=0}{=0}{=0}^{(x)} + \frac1{p^m} \ProbvalsBCnotpower{\geq 1}{=0}{=0}^{(x)} & \text{if } c = m\\[.5em]
			\frac{p-1}{p^m} \ProbvalsBCnotpower{=0}{=0}{=c-m}^{(x)} + \frac1{p^m} \ProbvalsBCnotpower{\geq 1}{=0}{=c-m}^{(x)} & \text{if } c > m. 
		\end{cases}
	\end{align*}
\end{lemma}

\section{Proofs of main theorems}
\label{sec:proofs}

\subsection{Proof of Theorem \ref{thm:rat_funcs}}

We begin with an elementary intermediate result about invertibility of certain matrices. Let $I_r$ denote the $r\times r$ identity matrix.

\begin{lemma}\label{lem:matrix}
    Let $N$ be an $r\times r$ matrix with entries in $[0,1)$ with at most one nonzero entry per row. Then $M = I_{r} - N$ is invertible. 
\end{lemma}

\begin{proof}
    The hypotheses ensure that the the matrix norm of $N = (n_{ij})$ induced by the $\ell^\infty$-norm is bounded,
    \[||N||_\infty = \max_{1 \leq i \leq r} \sum_{j=1}^r |n_{ij}| < 1.\]
    Since $||\cdot||_\infty$ is a submultiplicative matrix norm, any eigenvalue $\lambda$ of $N$ satisfies $\lambda \leq ||N||_\infty < 1$.
\end{proof}

Theorem \ref{thm:rat_funcs} follows from Lemma \ref{lem:probzero} (see also Remark \ref{rem:probBxz}) together with the following result.

\begin{theorem}\label{thm:rat_funcs_body}
    Fix $\ell,m,n$. There exists a rational function $R_{i,j,k}(t) \in \Q(t)$
    such that for all primes $p$ satisfying \pcond{$p \nmid \glm\gln\gmn$, $\Probzero(p) = \Probvals{=0}{=0}{=0}^{(x)}(p) = \Probvals{=0}{=0}{=0}^{(y)}(p) =\Probvals{=0}{=0}{=0}^{(z)}(p) = 1$,} and $\gcd(p-1,g_{\ell m}) = i$, $\gcd(p-1,g_{\ell n}) = j$, $\gcd(p-1,g_{mn}) = k$, we have
    \[\Prob_{\ell,m,n}(p) = R_{i,j,k}(p).\]    
\end{theorem}

\begin{proof}
    Fix a choice of $(i,j,k)$ and let $p$ be an arbitrary prime satisfying the hypotheses. We argue first that the $\ProbvalsBCnotpower{\geq 1}{=0}{=c}^{(x)}$ for $0 \leq c < n$ and $\ProbvalsBCnotpower{\geq 1}{=b}{=0}^{(x)}$ for $0 \leq b < m$ can be described by rational functions \textit{independent of $p$} (but depending on $(i,j,k)$). By Lemma \ref{lem:probA}, this will allow us to compute $\ProbA$, also as a rational function independent of $p$. Repeating the same process produces rational function expressions for $\ProbB$ and $\ProbC$. A similar argument for the $\Probvals{\geq 1}{=0}{=c}^{(x)}$ and $\Probvals{\geq 1}{=b}{=0}^{(x)}$ allows us to compute $\ProbAB$ (repeating this gives $\ProbBC$ and $\ProbAC$). Once all probabilities appearing in \eqref{eq:rho}, save for $\Prob$ itself, are described in terms of rational functions in $p$ that depend only on $(i,j,k)$, the result follows.

    By Lemmas \ref{lem:probA} and Lemma \ref{lem:auxiliary}, we have a system of $m+n-1$ linear equations in the $m+n-1$ variables $\ProbvalsBCnotpower{\geq 1}{=0}{=0}^{(x)}$, $\ProbvalsBCnotpower{\geq 1}{=0}{=c}^{(x)}$ for $0 < c < n$, and $\ProbvalsBCnotpower{\geq 1}{=b}{=0}^{(x)}$ for $0 < b < m$. Write this system as 
    \begin{equation}\label{eq:linearsystem}
        M\bv = \bu_{i,j,k},
    \end{equation}
    where the entries of $\bv$ correspond to the variables. Observe that the entries of $M$ live in $\Q(p)$: they depend only on the fixed values $\ell,m,n$ and not on $i,j,k$. Let us now examine the entries of $\bu_{i,j,k}$ more closely. By Lemma \ref{lem:probA} we have the following equation describing $\ProbvalsBCnotpower{\geq 1}{=0}{=0}^{(x)}$ (in the case $m \leq n$): 
    \[\ProbvalsBCnotpower{\geq 1}{=0}{=0}^{(x)} - \frac1{p^m} \ProbvalsBCnotpower{\geq 1}{=0}{=n-m}^{(x)} = \frac{p-1}{p^m} \ProbvalsBCnotpower{=0}{=0}{=n-m}^{(x)},\]
    so the corresponding entry in $\bu_{i,j,k}$ is 
    \[\frac{p-1}{p^m} \ProbvalsBCnotpower{=0}{=0}{=n-m}^{(x)} = \begin{cases}
        1 & m=n,\\
        1/i & n > m,
    \end{cases}\]
    by Lemma \ref{lem:ProbBCnotpower_easy}. 
    A similar analysis reveals that the constant terms in the relations coming from Lemma \ref{lem:auxiliary} are also rational functions in $p$ depending only on $i$ or $j$.

    For any fixed prime $p_0$, Lemma \ref{lem:matrix} ensures that the specialization $M(p_0)$ is invertible, thus $M$ must be invertible over $\Q(p)$. Therefore for any $i,j,k$, we can solve \eqref{eq:linearsystem} in $\Q(p)^{m+n-1}$, giving rational functions which specialize to $\ProbvalsBCnotpower{\geq 1}{=0}{=0}^{(x)}(p)$, $\ProbvalsBCnotpower{\geq 1}{=0}{=c}^{(x)}(p)$ for $0 < c < n$, and $\ProbvalsBCnotpower{\geq 1}{=b}{=0}^{(x)}(p)$ for $0 < b < m$, for all primes $p$ satisfying the hypotheses in the theorem statement. By Lemma \ref{lem:probA}, we can give similar descriptions of $\ProbA$ and $\Probvals{\geq 1}{=0}{=0}^{(x)}$. Repeating this process, we solve for $\ProbB$, $\Probvals{=0}{\geq 1}{=0}^{(y)}$, $\ProbC$, and $\Probvals{=0}{=0}{\geq 1}^{(z)}$, again as rational functions in $\Q(p)$ (depending only on $i,j,k$).
    
    Repeating this approach for the $\Probvals{\geq 1}{=0}{=c}^{(x)}$ for $0 < c < n$ and $\Probvals{\geq 1}{=b}{=0}^{(x)}$ for $0 < b < m$, we solve for these auxiliary probabilities as rational functions in $\Q(p)$ depending only on $i,j,k$. The same can be done to determine the relevant $\Probvals{=a}{\geq 1}{=c}^{(y)}$ and $\Probvals{=a}{=b}{\geq 1}^{(z)}$. Thus we obtain all the necessary values, as rational functions in $p$ depending only on $(i,j,k)$, to solve for $\ProbAB,\ProbAC,\ProbBC$ by Lemma \ref{lem:probAB}. Finally, evaluating \eqref{eq:rho} gives $\Prob$, and hence the expression for the rational functions $R_{i,j,k} \in \Q(p)$ in the statement of the theorem, completing the proof.
\end{proof}

The proof of Theorem \ref{thm:rat_funcs_body} also constitutes an effective algorithm to compute the rational functions $R_{i,j,k}(p)$ (for all but finitely many primes $p$). Sage code implementing this algorithm is available on GitHub \cite{github_gfedensity}. We compute a number of examples in \S \ref{sec:examples}.

\subsection{Proof of Theorem \ref{thm:stats}}

We now turn to analyzing the asymptotic behavior of $N_{\ell,m,n}^{\mathrm{loc}}(T)$. 

\begin{lemma}\label{lem:boundary}
    Let $\Omega_p = \left\{ (A,B,C) \in \Z_p^3 : \X_{A,B,C}(\Z_p) \neq \emptyset \right\}$. Then $\mu_p(\Omega_p) > 0$ and $\mu_p(\partial \Omega_p) = 0$.
\end{lemma}

\begin{proof}
    Consider $\Spec \Q_p[A,B,C,x,y,z] \simeq \mathbb{A}^6_{\Q_p}$ and the affine variety $V \subset \mathbb{A}^6_{\Q_p}$ given by
    \[V \colon Ax^\ell + By^m + Cz^n = 0.\]
    There is a projection morphism $\mathrm{pr} \colon V \to \Spec \Q_p[A,B,C] \simeq \mathbb{A}_{\Q_p}^3$.

    The set
    \[\{(A,B,C,x,y,z) \in \Z_p^6 : p \nmid x\}\]
    is a semialgebraic subset of $\Q_p^6$, and similarly for $p \nmid y$ and $p \nmid z$. This follows from the fact that $\Z_p \subset \Q_p$ is semialgebraic (see \cite[p.\ 185]{Ponomarev}) as are residue disks and their complements.
    
    In particular, 
    \[W = V(\Q_p) \cap \left\{ (A,B,C,x,y,z) \in \Z_p^6 : (x,y,z) \in \Z_p^3 \setminus (p\Z_p)^3 \right\}\]
    is semialgebraic, with $\mathrm{pr}(W) = \Omega_p$. At this point we essentially follow the proof of \cite[Lemma 3.9]{BBL}: by \cite[Theorem 1']{Ponomarev}, $\mathrm{pr}(W)$ is semialgebraic, and therefore $\Omega_p$ is measurable with $\mu_p(\partial \Omega_p) = 0$.

    To see that $\mu_p(\Omega_p) > 0$ for all $p$, it suffices to modify the proofs of Lemmas \ref{lem:p_divides_one}(i) and \ref{lem:probAyz} slightly to allow for $p \mid n$. Let $k > 2v_p(n)$ be an integer. Then for all $B,C \in \Z_p^\times$ such that $B \equiv -C \pmod{p^k}$, we have $-C/B \in (\Z_p/p^k\Z_p)^{\glm}$, so there exists a solution $(y,z) \in (\Z_p/ p^k\Z_p)^\times$ to $By^m + Cz^n \equiv 0 \pmod{p^k}$. Since $k > 2v_p(n)$, we can lift via Hensel's lemma in the $z$-coordinate to $(0,y,z) \in \Z_p^3 \setminus (p\Z_p)^3$. This immediately gives a lower bound of $\mu_p(\Omega_p) \geq \frac{p-1}{p^{k+1}}$.
\end{proof}

\begin{proof}[Proof of Theorem \ref{thm:stats}]
    For the first statement, suppose \expcond{$\ell,m,n$ are pairwise coprime}. In particular, at least one is odd, so \eqref{eq:gfe} always has real solutions. We aim to apply \cite[Proposition 3.4]{BBL} to see that the density of $[A:B:C] \in \P^2(\Z)$ for which $\X_{A,B,C}$ is everywhere locally soluble is computed by the product of $\Prob_{\ell,m,n}(p)$.

    We claim 
    \begin{equation}\label{eq:tech_cond}
        \lim_{M \to \infty} \limsup_{T \to \infty} \frac{\#\{(A,B,C) \in \Z^3 \cap [-T,T]^3 : \exists\ p > M \text{ s.t. } (A,B,C) \notin \Omega_p \}}{T^3} = 0.
    \end{equation}
    For each $M > 0$, the quantity inside the limit superior can be bounded above by
    \begin{equation}\label{eq:tail}
        \frac1{T^3} \sum_{p > M} (1 - \Prob(p)).
    \end{equation}
    Since we are taking a limit in $M$, we may assume $p$ is sufficiently large, so the hypotheses of Theorem \ref{thm:rat_funcs_body} are satisfied and $\Prob(p)$ is given uniformly by a \textit{single} rational function. Moreover, Lemma \ref{lem:probA} shows that $\ProbA = \ProbB = \ProbC = 1$ when $\ell,m,n$ are pairwise coprime, so $1 - \Prob(p) = O(1/p^2)$. Thus \eqref{eq:tail} is the tail of a convergent series, so \eqref{eq:tech_cond} goes to zero as claimed.

    Now we apply \cite[Proposition 3.4]{BBL}. Lemma \ref{lem:boundary} ensures the required boundary conditions on $\Omega_p$ are satisfied. For any bounded $\Psi \subset \Q^3$ with positive measure and boundary of measure zero, \cite[Equation 3.5]{BBL} follows from \eqref{eq:tech_cond}. Explicitly, if $\Psi \subseteq [-\lambda,\lambda]^3$, then we have
    \begin{align*}
         \lim_{M \to \infty} &\limsup_{T \to \infty} \frac{\#\{(A,B,C) \in \Z^3 \cap T\Psi : \exists\ p > M \text{ s.t. } (A,B,C) \notin \Omega_p \}}{T^3} \\
        & \leq \lim_{M \to \infty} \limsup_{T \to \infty} \frac{\#\{(A,B,C) \in \Z^3 \cap [\lambda T, \lambda T]^3 : \exists\ p > M \text{ s.t. } (A,B,C) \notin \Omega_p \}}{T^3} \\
        & = \lim_{M \to \infty} \lambda^3 \limsup_{T \to \infty} \frac{\#\{(A,B,C) \in \Z^3 \cap [\lambda T, \lambda T]^3 : \exists\ p > M \text{ s.t. } (A,B,C) \notin \Omega_p \}}{(\lambda T)^3} = 0
    \end{align*}
    by \eqref{eq:tech_cond}. Thus the quantity in the statement of Theorem \ref{thm:stats} is equal to
    \[\lim_{T \to \infty} \frac{\#\left\{P \in \P^2(\Q) : \Ht(P) \leq T,\ P \in \Omega_p \text{ for all } p\right\}}{\#\left\{P \in \P^2(\Q) : \Ht(P) \leq T\right\}} = \prod_p \mu_p(\Omega_p) = \prod_p\Prob(p).\]
    Note that it was essential that $1- \Prob(p) = O(1/p^2)$ for all $p$. This is not the case when \expcond{$\ell,m,n$ are not pairwise coprime}. 

    Suppose now that \expcond{$\ell,m,n$ are not pairwise coprime}. Without loss of generality, \expcond{suppose $\gmn > 1$.} Fix $M > 0$ and define
    \[\Omega_p' = \begin{cases}
        \Omega_p & p \equiv 1 \pmod{\gmn} \text{ and } p < M,\\
        \Z_p^3 & \text{otherwise}.
    \end{cases}\]
    Another application of \cite[Proposition 3.4]{BBL} shows\footnote{This time \cite[Equation 3.5]{BBL} holds trivially since $\Omega_p' = \Z_p^3$ for all but finitely many primes $p$.}
    \[\lim_{T \to \infty} \frac{\#\left\{P \in \P^2(\Q) : \Ht(P) \leq T,\ P \in \Omega_p' \text{ for all } p\right\}}{\#\left\{P \in \P^2(\Q) : \Ht(P) \leq T\right\}} = \prod_p \mu_p(\Omega_p') = \prod_{\substack{p < M \\ p \equiv 1 \ (\gmn)}} \Prob(p),\]
    which is an upper bound for the desired limit superior.
    
    Among sufficiently large primes $p \equiv 1 \pmod{\gmn}$, we can give a uniform upper bound 
    \[\Prob(p) \leq 1 - \frac{\gmn - 1}{\gmn p} + O(1/p^2)\]
    using Lemma \ref{lem:probA} and \eqref{eq:rho}.  By an application of Mertens' product theorem for primes in arithmetic progressions \cite{Williams}, we have
    \begin{equation}
    \label{eq:mertens_step}
        \prod_{\substack{p < M \\ p \equiv 1 \ (\gmn)}} \Prob(p) \ll \left(\prod_{\substack{p < M \\ p \equiv 1 \ (\gmn)}} \left( 1 - \frac1p\right)\right)^{(\gmn - 1)/\gmn} \ll \left(\frac1{\log M}\right)^{\frac{\gmn - 1}{\gmn \varphi(\gmn)}}
    \end{equation}
    where $\varphi$ denotes Euler's totient function. This approaches zero as $M \to \infty$, finishing the proof. 
\end{proof}

\subsection{Number fields}
\label{subsec:NFs}

Our main results, Theorems \ref{thm:rat_funcs} and \ref{thm:stats}, may be extended to number fields as follows. Let $K/\Q$ be a number field and $\O_K$ its ring of integers. For a prime ideal $\p \subset \O_K$, we let $K_\p$ denote the completion, with ring of integers $\Op$ and residue field $\F_\p = \Op/\p$. The norm of $\p$ is $N_K(\p) = \# \F_\p$.

For $\ell,m,n \geq 1$ and $A,B,C \in \O_K$, primitive $\Op$-solutions to \eqref{eq:gfe} correspond to $\Op$-points\footnote{Caution: the same is not quite true for $\O_K$-points if the class group is nontrivial.} on the stacky curve $\X_{A,B,C}$, defined as in \S \ref{subsec:GFEs}. Since detecting $\Op$-points on $\X_{A,B,C}$ boils down to studying valuations of the coefficients and lifting solutions over $\F_\p$, the intermediate results of \S \ref{subsec:padic_pts} carry over to this setting unchanged.

Defining the analogous local solubility density,
\[\ProbK_{\ell,m,n}(\p) = \mu_\p \left( \left\{ (A,B,C) \in \Op^3 : \X_{A,B,C}(\Op) \neq \emptyset \right\}\right),\]
the techniques of \S \ref{sec:probs} and \S \ref{sec:relations} can be used to define auxiliary probabilities and find relations, with $N_K(\p)$ playing the role of (the size of) $p$. Moreover, since e.g.\ $\gcd(\Np - 1, \glm)$ is controlled by how $\p$ splits in a cyclotomic extension of $K$, we can more cleanly describe the resulting rational functions as depending on the Frobenius of $\p$ in the extension. 

Let $\zeta$ be a primitive $\lcm(\glm,\gln,\gmn)$-th root of unity. The extension $K(\zeta)/K$ is Galois and cyclic, so for $\p$ unramified (i.e.\ with $\Np$ coprime to $\glm\gln\gmn$) the Frobenius $\Frobp$ is an element of the Galois group $\mathrm{Gal}(K(\zeta)/K)$. Moreover, $\Frobp$ determines the value of $\Np$ modulo $\lcm(\glm,\gln,\gmn)$, hence also those of $\gcd(\Np - 1, \glm)$, $\gcd(\Np - 1, \gln)$, and $\gcd(\Np - 1, \gmn)$. Thus, following the proof of Theorem \ref{thm:rat_funcs_body}, we may extend Theorem \ref{thm:rat_funcs} to the number field setting.

\begin{theorem}\label{thm:rat_funcs_NFs}
	With the notation above, for each $\tau \in \mathrm{Gal}(K(\zeta)/K)$ there exists a rational function $R_\tau(t) \in \Q(t)$ such that for all but finitely many primes $\p \subset \O_K$, we have $\Prob^{(K)}_{\ell,m,n}(\p) = R_{\Frobp}(\Np)$.
\end{theorem}

For the counting results, we define $N^{\mathrm{loc}}_{\ell,m,n}(K;T)$ to count $P = [A:B:C] \in \P^2(K)$ with bounded height (see e.g.\ \cite[\S 3.2]{BBL} for the definition of height in this setting) for which $\X_{A,B,C}(\Op) \neq \emptyset$ for all $\p$. The methods of \cite{BBL} work over number fields and Lemma \ref{lem:boundary} generalizes readily to this setting. Modifying the proof of Theorem~\ref{thm:stats} appropriately, we obtain the following results.

\begin{theorem}\label{thm:stats_NFs}
    For a number field $K$, the asymptotic growth of $N^{\mathrm{loc}}_{\ell,m,n}(K;T)$ can be described as follows.
    \begin{enumerate}[label = (\roman*)]
        \item If \expcond{$\ell,m,n$ are pairwise coprime},
        \[\frac{N^{\mathrm{loc}}_{\ell,m,n}(K;T)}{\#\left\{P \in \P^2(K) : \Ht(P) \leq T\right\}} \sim \prod_{\p} \ProbK_{\ell,m,n}(\p) > 0.\]
        \item If instead \expcond{$\ell,m,n$ are not pairwise corpime},
        \[\limsup_{T \to \infty} \frac{N^{\mathrm{loc}}_{\ell,m,n}(K;T)}{\#\left\{P \in \P^2(K) : \Ht(P) \leq T\right\}} = 0.\]
    \end{enumerate}
\end{theorem}

Note that for the last step in the proof of Theorem \ref{thm:stats_NFs}(ii), the invocation of \cite{Williams} for $p \equiv 1 \pmod \gmn$ to achieve \eqref{eq:mertens_step} needs to be replaced by an application of \cite[Theorem A]{APKK_Mertens} for prime ideals $\p$ with $\Frobp$ equal to the identity element of $\mathrm{Gal}(K(\zeta)/K)$.

\begin{remark}\label{rem:prodNloc_bigsmall}
    As $K$ varies, the density of $P \in \P^2(K)$ for which $\X_P$ is everywhere locally soluble can be arbitrarily close to 0 or 1. See the discussion in \S \ref{subsubsec:numericsNFs}, in particular Proposition \ref{prop:liminf_limsup_ProbK}.
\end{remark}

\section{Explicit examples}
\label{sec:examples}

Sage code demonstrating the examples worked out in this section is available at \cite{github_gfedensity}.

\subsection{Diagonal plane curves}
\label{subsec:nnn}

When $\ell=m=n$, \eqref{eq:gfe} describes a diagonal plane curve of degree $n$, as studied in \cite{HirakawaKanamura}. In particular, $\X_{A,B,C}(\Z_p) = X_{A,B,C}(\Z_{p}) = X_{A,B,C}(\Q_p)$. For $n \in \{2,3,4,5\}$, Theorem \ref{thm:rat_funcs} produces the following density functions: 
\begin{align}
    \label{eq:222}\Prob_{2,2,2}(p) &= 
        1 - \frac{3p}{2(p+1)^2}
        &\pcond{(p \neq 2)}\\
    \label{eq:333}\Prob_{3,3,3}(p)    &= 
    \begin{cases}
        1 - \frac{2p}{p^2 + p + 1},
        & p \equiv 1 \pmod{3} \\ 
        1 - \frac{6p^3}{{\left(p^{2} + p + 1\right)}^{3}},
        & p \equiv 2 \pmod{3}
    \end{cases}
    & \pcond{(p \neq 3)}\\
    \label{eq:444}\Prob_{4,4,4}(p) &= \begin{cases}
        1 - \frac{3  {\left(3  p^{4} + 3  p^{3} + 5  p^{2} + 3  p + 3\right)} p}{4  {\left(p^{2} + 1\right)}^{2} {\left(p + 1\right)}^{2}},
        & p \equiv 1 \pmod{4}\\
        1 - \frac{3  {\left(p^{4} + p^{3} + 3  p^{2} + p + 1\right)} p}{2  {\left(p^{2} + 1\right)}^{2} {\left(p + 1\right)}^{2}}, 
        & p \equiv 3 \pmod{4}
    \end{cases}
    & \pcond{(p > 31)}\\
    \label{eq:555}\Prob_{5,5,5}(p) &= \begin{cases}
        1 - \frac{6  {\left(2  p^{4} + 2  p^{3} + p^{2} + 2  p + 2\right)} {\left(p^{2} + 1\right)} p}{5  {\left(p^{4} + p^{3} + p^{2} + p + 1\right)}^{2}},
        & p \equiv 1 \pmod{5} \\
        1 - \frac{6  {\left(p^{2} + 1\right)} p^{3}}{{\left(p^{4} + p^{3} + p^{2} + p + 1\right)}^{2}},
        & p \not\equiv 0,1 \pmod{5}
    \end{cases}
    & \pcond{(p > 131)}
\end{align}
The densities \eqref{eq:222} and \eqref{eq:333} recover those presented in \cite[Theorems 1.3(1) and 1.4(1)]{HirakawaKanamura}. Their methods generalize to diagonal hypersurfaces in $\P^n$. Another description of local solubility in this case is given in \cite[Proposition 2.7]{KPSS_locsolfermat} in terms of the valuations of the coefficients modulo $n$. Writing $A = p^aA_0$, $B = p^bB_0$, $C = p^cC_0$ for $A_0,B_0,C_0 \in \Z_p^\times$ and $a \equiv \overline{a} \pmod{n}$ for $\overline{a} \in \{0, \ldots, n-1\}$ (similarly define $\overline{b}, \overline{c}$), we have 
\begin{itemize}
    \item if $\overline{a}=\overline{b}=\overline{c}$ then $\X_{A,B,C}(\Z_p) \neq \emptyset$;
    \item if $\overline{a} = \overline{b} \neq \overline{c}$ then $\X_{A,B,C}(\Z_p) \neq \emptyset$ if and only if $-B/A \in (\F_p^\times)^n$;
    \item if $\overline{a},\overline{b},\overline{c}$ are distinct then $\X_{A,B,C}(\Z_p) = \emptyset$.
\end{itemize}
These conditions agree with the points of types I, II, and III, respectively, defined in \cite[Definition 2.1]{HirakawaKanamura}.

The validity ranges shown are taken from \cite[Remark 2.4]{HirakawaKanamura} (see also \cite[Definition 2.6]{KPSS_locsolfermat}): the formulae are valid for primes $p \nmid n$ satisfying $p \geq (n-1)^2(n-2)^2$ or $\gcd(p-1,n) = 1$. Notice that if we apply Theorem \ref{thm:rat_funcs_body} as stated, using Remark \ref{rem:improvements_to_pt_ct_bound} to ensure the hypotheses are satisfied, we obtain slightly weaker validity bounds. For instance, when $n=3$ we get $p > 7$ or $p \equiv 2 \pmod{3}$, so $p=7$ is missing. This is a consequence of requiring $\Probvals{=0}{=0}{=0}^{(x)} = 1$, which turns out to be slightly stronger than necessary in the special case $\ell=m=n$.

To illustrate why this is the case, consider $\Probvals{=2n}{=n}{=0} = \Probvals{=n}{=0}{=0}^{(x \text{ or } y)}$. For any $A_0,B_0,C_0 \in \Z_p^\times$, if $-C_0/B_0 \in (\F_p^\times)^n$, then $\X_{p^nA_0,B_0,C_0}(\Z_p ; p \nmid y) \neq \emptyset$. On the other hand, if $-C_0/B_0 \notin (\F_p^\times)^n$, then $\X_{p^nA_0,B_0,C_0}(\Z_p ; p \nmid x \text{ or } p \nmid y) = \X_{A_0,B_0,C_0}(\Z_p; p \nmid x)$. In this case, $\X_{A_0,B_0,C_0}(\F_p; x = 0) = \emptyset$, so it suffices to $\X_{A_0,B_0,C_0}(\F_p) \neq \emptyset$. This gives sharper bounds on $p$ for which $\Probvals{=2n}{=n}{=0}(p) = 1$ than identifying $\Probvals{=2n}{=n}{=0} = \Probvals{=0}{=0}{=0}^{(x)}(p)$.

For general $(\ell,m,n)$, $\Probvals{=0}{=0}{=0}^{(x)}$ shows up in auxiliary probabilities, e.g.\ $\Probvals{\geq 1}{=0}{=m}$ when $m < n$, which do not for $\ell=m=n$. The minor tradeoff we make in assuming  $\Probvals{=0}{=0}{=0}^{(x)} = 1$ streamlines our approach, at the (potential) expense of some minor optimization of the validity range.

\subsection{\texorpdfstring{$(\boldsymbol{\ell},\boldsymbol{m},\boldsymbol{n}) = (\boldsymbol{2},\boldsymbol{2},\boldsymbol{n})$}{(l,m,n) = (2,2,n)}}
\label{subsec:22n}

Suppose $n \geq 2$. We can give a closed form expression for $\Prob_{2,2,n}(p)$ for odd $p$, 
\begin{equation}\label{eq:22nodd}
    \Prob_{2,2,n}(p) = 1 - \frac{f_{2,2,n}(p)}{g_{2,2,n}(p)},
\end{equation}
where $f_{2,2,n}$ and $g_{2,2,n}$ are given in Table \ref{tab:data}. The cases of odd and even $n$ are treated separately in \S \ref{subsubsec:22n_n_odd} and \S \ref{subsubsec:22n_n_even}. In \S \ref{subsubsec:22n_p=2} we address the computation of $\Prob_{2,2,n}(2)$ for $n$ odd.

A related probability was given in \cite[Remark 7.2]{DRKKRSW_22n}, under the assumption that $n \geq 3$ is odd and $p \neq 2$:
\[\mu_p\left( \left\{ (B,C) \in \Z_p^2 : \exists\ (x,y,z) \in \Z_p^3 \setminus (p\Z_p)^3 \text{ s.t. } x^2 + By^2 - Cz^n = 0 \right\} \right) = 1 - \frac{p^n + p^{n-2} + 2p^{n-3} - p + 2}{2p^{n-3}(p+1)(p^3 + p^2 + p + 1)}.\]

\subsubsection{\texorpdfstring{$n \geq 3$ odd}{n at least 3 and odd}}
\label{subsubsec:22n_n_odd}

We briefly sketch how to arrive at \eqref{eq:22nodd} for $p \neq 2$. An application of Lemma \ref{lem:intermediate} shows
\[\Probzero = \ProbA = \ProbB = 1.\]
By Lemma \ref{lem:probC_l=m}, we have
\[\ProbC = \left(1-\frac1{p^2}\right)^{-1}\left( \frac{p-1}{2p} + \frac{p-1}{p^2} \right) = \frac{p+2}{2p + 2}\]
for $p \neq 2$.
Next, we use Lemma \ref{lem:auxiliary} to deduce the following.
\begin{align}
\label{eq:aux_22n_n_odd}
    \Probvals{\geq 1}{=0}{=c}^{(x)} &= \begin{cases}
        \frac{p-1}{p} + \frac1p \Probvals{\geq 1}{=1}{=0}^{(x)}
            = \frac{2p^{n+1} + p^n - 2p^{n-1} + p + 2}{2(p+1)p^{n}}
            & c = 1 \\
        \frac{p-1}{p^2} + \frac1{p^2} \Probvals{\geq 1}{=0}{=0}^{(x)} 
            = \frac1p
            & c = 2 \\
        \frac{p^{c-1} - 1}{2p^{c-1}(p+1)} + \frac1{p^{c-1}} \Probvals{\geq 1}{=0}{=1}^{(x)} & 2 < c < n \text{ odd} \\
        \frac{p^{c-2} - 1}{2p^{c-2}(p+1)} + \frac1{p^{c-2}} \Probvals{\geq 1}{=0}{=2}^{(x)} & 2 < c < n \text{ even}
    \end{cases}\\
    \nonumber \Probvals{\geq 1}{=1}{=0}^{(x)} &= \frac{p-1}{2p} + \frac1p \Probvals{\geq 1}{=0}{=n-1}^{(x)} 
        = \frac{p^n + p + 2}{2(p + 1)p^{n-1}}\\
    \nonumber \Probvals{=0}{=1}{\geq 1}^{(z)} &= \Probvals{=1}{=0}{\geq 1}^{(z)} = 1. 
\end{align}
Solving the linear system given by relations among $\Probvals{\geq 1}{=0}{=1}^{(x)}$, $\Probvals{\geq 1}{=1}{=0}^{(x)}$, and $\Probvals{\geq 1}{=0}{=n-1}^{(x)}$ produces the values stated above in \eqref{eq:aux_22n_n_odd}.

Using Lemma \ref{lem:probAB}, it is straightforward to compute
\[\ProbAC = \ProbBC = \frac{2p^{n+1} + 3p^n + 2p^{n-1} + 3p^{n-2} + 2p^{n-3} + p + 2}{2\left(p^2 + 1\right)\left(p + 1\right)^2 p^{n-3}}.\]
A somewhat more tedious computation involving several geometric series identities produces
\[\ProbAB = \frac{p^{3n} +p^{3n-2} + 2p^{2n+1} + 8p^{2n} - 4p^{2n-2}  - p^{n+2} -4p^{n+1} + p^n + 4p^{n-1} - 2p^{n-2} -2p - 4}{2\left(p + 1\right)^2\left(p^n + 1\right)\left(p^n - 1\right)p^n}.\]
Finally, we compute $\Prob_{2,2,n}(p)$ using \eqref{eq:rho} and record it in Table \ref{tab:data}.

\subsubsection{\texorpdfstring{$n \geq 4$ even}{n at least 4 and even}}
\label{subsubsec:22n_n_even}

Suppose $n \geq 4$ is even ($n=2$ was discussed in \S \ref{subsec:nnn}) and $p \neq 2$. We claim that $\Probvals{=0}{=0}{=0}^{(x)} = \Probvals{=0}{=0}{=0}^{(y)} = \Probvals{=0}{=0}{=0}^{(z)} = 1$. For $p > 5$, this follows from Lemma \ref{lem:probzero}. For $p \in \{3,5\}$ this can be seen by exhausting all possible quadratic forms $Ax^2 + By^2$ for $A,B \in \F_p^\times$ and using a straightforward Hensel lifting argument. 

The same argument as \S \ref{subsubsec:22n_n_odd} shows $\ProbC = \frac{p+2}{2p + 2}$. Computing $\ProbA = \ProbB$ is slightly more involved, since $\glm = \gln = 2$. By Lemmas \ref{lem:probA} and \ref{lem:auxiliary}, we have
\[
    \ProbvalsBCnotpower{\geq 1}{=0}{=c}^{(x)} = \begin{cases}
        \frac{p-1}{2p^2} + \frac1{p^2}\ProbvalsBCnotpower{\geq 1}{=0}{=n-2}^{(x)} 
            = \frac{p^n + p^2 - 2}{2\left(p^{n+1} + p^n - p - 1\right)} 
            & c = 0,\\
        \frac{p-1}{p^2} + \frac1{p^2}\ProbvalsBCnotpower{\geq 1}{=0}{=0}^{(x)} 
            = \frac{2p^n - p^{n-2}- 1}{2\left(p^{n+1} + p^n - p - 1\right)}
            & c = 2,\\
        \frac{p-1}{2p^2} + \frac1{p^2}\ProbvalsBCnotpower{\geq 1}{=0}{=c-2}^{(x)} & 2 < c < n \text{ even}.
    \end{cases}
\]
From the relations we deduce a formula for $\ProbvalsBCnotpower{\geq 1}{=0}{=n-2}^{(x)}$ 
from which we can solve for $\ProbvalsBCnotpower{\geq 1}{=0}{=2}^{(x)}$. Then from Lemma \ref{lem:probA} we have
\[\ProbA = \ProbB = \frac{2p^{n+1} + 3p^n + p^2 - 2p - 4}{4\left(p^{n+1} + p^n - p - 1\right)}.\]

Next, we need to compute the relevant auxiliary probabilities using Lemma \ref{lem:auxiliary}. The approach is the same as in \S \ref{subsubsec:22n_n_odd}.
\begin{align*}
    \Probvals{\geq 1}{=0}{=c}^{(x)} &= \begin{cases}
        \frac{p-1}{2p} + \frac1p \Probvals{\geq 1}{=1}{=0}^{(x)} 
            = \frac{p^{n+1} + p^n - p^{n-1} - 1}{2\left(p^{n+1} + p^n - p - 1\right)}
            & c = 1 \\
        \frac{p-1}{p^2} + \frac1{p^2} \Probvals{\geq 1}{=0}{=0}^{(x)} 
            = \frac{4p^{n+2} + 2p^{n+1} - p^n + p^{n-2} - p^{n-4} - 4p^2 - 2p + 1}{4p^2 \left(p^{n+1} + p^{n} - p - 1\right)}
            & c = 2 \\
        \frac{p^{c-1} - 1}{2p^{c-1}(p+1)} + \frac1{p^{c-1}} \Probvals{\geq 1}{=0}{=1}^{(x)} & 2 < c < n \text{ odd} \\
        \frac{p^{c-2} - 1}{2p^{c-2}(p+1)} + \frac1{p^{c-2}} \Probvals{\geq 1}{=0}{=2}^{(x)} & 2 < c < n \text{ even}
    \end{cases}\\
    \Probvals{\geq 1}{=1}{=0}^{(x)} &= \frac{p-1}{2p} + \frac1p \Probvals{\geq 1}{=0}{=n-1}^{(x)} = \frac{p^{n+1} + p^2 - p - 1}{2\left(p^{n+1}+ p^n - p - 1\right)}\\
    \Probvals{=0}{=1}{\geq 1}^{(z)} &= \Probvals{=1}{=0}{\geq 1}^{(z)} = \frac12
\end{align*}
Finally, these can be combined via Lemma \ref{lem:probAB} to give
\begin{align*}
    \ProbAB &= \frac{p^{3n+2} +p^{3n} + 2p^{2n+3} + p^{2n+2} - 2p^{2n} - p^{n+4} -2p^{n+3} -2p^{n+2} - 2p^{n+1} - p^n + p^4 + 2p + 2}{\left(p + 1\right)^2 \left(p^{2n} - 1\right)\left(p^n - 1\right)},\\
    \ProbAC = \ProbBC &= \frac{2p^{n+4} + 4p^{n+3} + 4p^{n+2} + 4p^{n+1} + 3p^n - 4p^3 - 5p^2 - 4p - 4}{4\left(p^3 + p^2 + p + 1\right)\left(p^{n+1} + p^n - p - 1\right)}.
\end{align*}
We then compute $\Prob$ by \eqref{eq:rho} and record it in Table \ref{tab:data}.

\subsubsection{\texorpdfstring{$p=2$}{p=2}}
\label{subsubsec:22n_p=2}

We restrict our attention to $n \geq 3$ odd, where some of the relevant probabilities are unchanged: $\Probzero(2) = \ProbA(2) = \ProbB(2) = 1$. When $n \geq 4$ is even, these are strictly less than 1, but a similar analysis can be carried out to determine $\Prob_{2,2,n}(2)$.

To address $\ProbC$, consider the following subsets of $\Z/8$:
\begin{align*}
    \left\{ Ax^2 + By^2 : x,y \in \Z/8^\times \right\} &= \begin{cases}
        \{ 2\} & B/A \equiv 1 \pmod{8},\\
        \{ 4\} & B/A \equiv 3 \pmod{8},\\
        \{ 6\} & B/A \equiv 5 \pmod{8},\\
        \{ 0\} & B/A \equiv 7 \pmod{8},
    \end{cases} \\
    \left\{ Cz^n : z \in \Z/8^\times \right\} &= \begin{cases}
        \{1,3,5,7\} & v_2(C) = 0,\\
        \{2,6\} & v_2(C) = 1,\\
        \{4\} & v_2(C) = 2,\\
        \{0\} & v_2(C) \geq 3.
    \end{cases}
\end{align*}
Note that the latter depends only on the 2-adic valuation of $C$ since $n \geq 3$ is odd.

This allows us to deduce, e.g.\ $\Probvals{=0}{=0}{=1}^{(z)}(2) = \frac12$. On the other hand, if we do not ask for $z \in \Z_p^\times$, we find $\Probvals{=0}{=0}{=1}(2) = \Probvals{=0}{=0}{=1}^{(x)}(2) = \frac34$, since we could take $z = 4$ and lift a solution to $Ax^2 + By^2 \equiv 0 \pmod{8}$ whenever $B/A \equiv 7 \pmod{8}$. A similar analysis yields 
\begin{align*}
    \Probvals{=0}{=0}{=c}^{(x)}(2) = \Probvals{=0}{=0}{=c}^{(y)}(2) &= \begin{cases}
        1 & c=0,\\
        \frac34 & c=1,\\
        \frac12 & c= 2,\\
        \frac14 & c\geq 3,
    \end{cases} \\
    \Probvals{=0}{=0}{=c}^{(z)}(2) &= \begin{cases}
        1 & c \text{ even},\\
        \frac12 & c = 1,\\
        \frac34 & c \geq 3 \text{ and } c \text{ odd},
    \end{cases}\\
    \Probvals{=0}{=0}{=c}(2) &= \begin{cases}
        1 & c \text{ even},\\
        \frac34 & c \text{ odd}.
    \end{cases}     
\end{align*}
From these we deduce
\begin{align*}
    \Probvals{=0}{=0}{\geq 1}^{(x)}(2) = \Probvals{=0}{=0}{\geq 1}^{(y)}(2) &= \frac58,\\
    \Probvals{=0}{=0}{\geq 1}^{(z)}(2) &= \frac{17}{24},\\
    \ProbC(2) = \Probvals{=0}{=0}{\geq 1}(2) &= \frac56.
\end{align*}
This provides enough information to set up and solve the relations given in Lemma \ref{lem:auxiliary}, and hence determine $\Prob_{2,2,n}(2)$. We record the first few values below. 
\[\Prob_{2,2,n}(2) = \begin{cases}
    \frac{3853}{4410} & n=3, \\
    \frac{174677}{214830} & n=5, \\
    \frac{5459633}{6880860} & n=7, \\
    \frac{38572339}{48933360} & n=9, \\
    \frac{22180686643}{28185716160} & n=11.
\end{cases}\]

\subsection{\texorpdfstring{$(\boldsymbol{\ell},\boldsymbol{m},\boldsymbol{n}) = (\boldsymbol{3},\boldsymbol{3},\boldsymbol{2})$}{(l,m,n) = (3,3,2)}}
\label{subsec:332}

When $(\ell,m,n) = (3,3,2)$, we have
\begin{equation}\label{eq:332}
    \Prob_{3,3,2}(p) = \begin{cases}
        1 - \frac{2  p^{11} + 4  p^{10} + 14  p^{9} + 10  p^{8} + 14  p^{7} + 8  p^{6} + 4  p^{5} + 4  p^{4} - 8  p^{2} - 6  p - 6}
        {3  {\left(p^{5} + p^{4} + p^{3} + p^{2} + p + 1\right)} {\left(p^{2} + p + 1\right)}^{2} {\left(p^{2} + 1\right)} p}
        & p \equiv 1 \pmod{3}, \\
        1 - \frac{4  p^{9} + 2  p^{8} + 4  p^{7} + 2  p^{6} + 2  p^{4} - 2  p^{2} - 2  p - 2}
        {{\left(p^{5} + p^{4} + p^{3} + p^{2} + p + 1\right)} {\left(p^{2} + p + 1\right)}^{2} {\left(p^{2} + 1\right)} p}
        & p \equiv 2 \pmod{3},\\
        \frac{424871}{461370} & p=3.
    \end{cases}
\end{equation}
The hypotheses of Theorem \ref{thm:rat_funcs_body} are satisfied for $p \neq 3$, by Lemma \ref{lem:intermediate}. We handle $p=3$ below.

\subsubsection{\texorpdfstring{$p=3$}{p=3}}
\label{subsubsec:332_p=3}

To compute $\Prob_{3,3,2}(3)$, we have by Lemmas \ref{lem:intermediate} and \ref{lem:p_divides_one}
\[\Probzero(3) = \ProbA(3) = \ProbB(3) = 1.\]
To address $\ProbC(3)$ (and its variants), we consider the congruence
\begin{equation}\label{eq:332_p=3_cong}
    Ax^3 + By^3 + Cz^2 \equiv 0 \pmod{27},
\end{equation}
when $A,B \in (\Z/27\Z^\times)^3$:
\begin{itemize}
    \item if $v_p(C) = 1$, there exist $x,y,z \in (\Z/27\Z^\times)^3$ satisfying \eqref{eq:332_p=3_cong} if and only if $A \not\equiv \pm B \pmod{9}$;
    \item if $v_p(C) \geq 2$, there exist $x,y,z \in (\Z/27\Z^\times)^3$ satisfying \eqref{eq:332_p=3_cong} if and only if $A \equiv \pm B \pmod{9}$.
\end{itemize}
These are easily verified by elementary means or direct computation. Any solution to \eqref{eq:332_p=3_cong} can be lifted to a $\Z_p$-solution by Hensel's lemma. It follows that
\begin{align*}
    \Probvals{=0}{=0}{=c}^{(x)}(3) = \Probvals{=0}{=0}{=c}^{(y)}(3) &= \begin{cases}
        1 & c = 1,\\
        \frac13 & c \geq 2.
    \end{cases}\\
    \Probvals{=0}{=0}{=c}^{(z)}(3) &= \begin{cases}
        1 & c \equiv 0 \pmod{3}\\
        \frac23 & c = 1,\\
        1 & c > 1 \text{ and } c \equiv 1 \pmod{3},\\
        \frac13 & c \equiv  2 \pmod{3},
    \end{cases}\\
    \Probvals{=0}{=0}{=c}(3) &= \begin{cases}
        1 & c \equiv 0 \pmod{3},\\
        1 & c \equiv 1 \pmod{3},\\
        \frac13 & c \equiv 2.
    \end{cases}
\end{align*}
From this we deduce
\begin{align*}
    \Probvals{=0}{=0}{\geq 1}^{(x)}(3) = \Probvals{=0}{=0}{\geq 1}^{(y)}(3) &= \frac{13}{18}, \\
    \Probvals{=0}{=0}{\geq 1}^{(z)}(3) &= \frac{73}{117}, \\
    \Probvals{=0}{=0}{\geq 1}(3) &= \frac{11}{13}.
\end{align*}
With this, we can solve for all the relevant probabilities, finding $\Prob_{3,3,2}(3) = \frac{424871}{461370}$.

\subsection{\texorpdfstring{$(\boldsymbol{\ell},\boldsymbol{m},\boldsymbol{n}) = (\boldsymbol{6},\boldsymbol{10},\boldsymbol{15})$}{(l,m,n) = (6,10,15)}}
\label{subsec:6_10_15}

In this case the pairwise gcds $(\glm, \gln, \gmn) = (2,3,5)$ are all greater than $1$ and pairwise coprime. For \pcond{$p > 467$}, we can compute $\Prob_{6,10,15}(p)$ via Theorem \ref{thm:rat_funcs}. Setting
\[i = \gcd(p-1,2),\ j = \gcd(p-1,3),\ k = \gcd(p-1,5)\]
and recognizing that $i = 2$ for all $p > 2$, we can compute the rational functions $R_{2,j,k}(p)$ explicitly for $(j,k) \in \{(1,1),(3,1),(1,5),(3,5)\}$. For ease of presentation, we write
\[R_{2,j,k}(p) = 1 - \frac{S_{j,k}(p)}{jk D(p)},\]
where $S_{j,k}(p)$ and $D(p)$ are the rational functions given below.
\begin{dgroup*}
    \begin{dmath}
        D(p) = 8  \left(p^{13} + p^{12} + 2  p^{11} + 2  p^{10} + 3  p^{9} + 3  p^{8} + 3  p^{7} + 3  p^{6} + 3  p^{5} + 3  p^{4} + 2  p^{3} + 2  p^{2} + p + 1\right)\left(p^{12} + p^{9} + p^{6} + p^{3} + 1\right)\left(p^{12} - p^{9} + p^{6} - p^{3} + 1\right)\left(p^{11} + p^{10} + p^{9} + p^{8} + p^{7} + p^{6} + p^{5} + p^{4} + p^{3} + p^{2} + p + 1\right)\left(p^{8} - p^{6} + p^{4} - p^{2} + 1\right)\left(p^{2} + p + 1\right) p^{15}
    \end{dmath}
    \begin{dmath}
        S_{1,1}(p) = 
        4  p^{72} + 8  p^{71} + 52  p^{70} + 56  p^{69} + 104  p^{68} + 108  p^{67} + 188  p^{66} + 184  p^{65} + 288  p^{64} + 280  p^{63} + 408  p^{62} + 392  p^{61} + 540  p^{60} + 512  p^{59} + 652  p^{58} + 584  p^{57} + 752  p^{56} + 716  p^{55} + 856  p^{54} + 788  p^{53} + 936  p^{52} + 868  p^{51} + 988  p^{50} + 920  p^{49} + 1016  p^{48} + 920  p^{47} + 1036  p^{46} + 980  p^{45} + 1028  p^{44} + 944  p^{43} + 988  p^{42} + 880  p^{41} + 932  p^{40} + 820  p^{39} + 860  p^{38} + 760  p^{37} + 772  p^{36} + 648  p^{35} + 684  p^{34} + 572  p^{33} + 548  p^{32} + 428  p^{31} + 448  p^{30} + 336  p^{29} + 344  p^{28} + 228  p^{27} + 196  p^{26} + 128  p^{25} + 108  p^{24} + 20  p^{23} - 8  p^{22} - 40  p^{21} - 68  p^{20} - 92  p^{19} - 108  p^{18} - 116  p^{17} - 148  p^{16} - 128  p^{15} - 132  p^{14} - 84  p^{13} - 104  p^{12} - 84  p^{11} - 76  p^{10} - 40  p^{9} - 60  p^{8} - 40  p^{7} - 32  p^{6} - 16  p^{5} - 24  p^{4} - 8  p^{3} - 8  p^{2} - 8  p - 8
    \end{dmath}
    \begin{dmath}
        S_{3,1}(p) = 
        28  p^{72} + 56  p^{71} + 172  p^{70} + 200  p^{69} + 344  p^{68} + 388  p^{67} + 612  p^{66} + 648  p^{65} + 928  p^{64} + 968  p^{63} + 1304  p^{62} + 1320  p^{61} + 1716  p^{60} + 1696  p^{59} + 2052  p^{58} + 1976  p^{57} + 2416  p^{56} + 2372  p^{55} + 2744  p^{54} + 2636  p^{53} + 3032  p^{52} + 2892  p^{51} + 3220  p^{50} + 3088  p^{49} + 3352  p^{48} + 3136  p^{47} + 3436  p^{46} + 3276  p^{45} + 3436  p^{44} + 3192  p^{43} + 3324  p^{42} + 3040  p^{41} + 3156  p^{40} + 2860  p^{39} + 2956  p^{38} + 2632  p^{37} + 2676  p^{36} + 2336  p^{35} + 2372  p^{34} + 2028  p^{33} + 1972  p^{32} + 1620  p^{31} + 1616  p^{30} + 1288  p^{29} + 1272  p^{28} + 940  p^{27} + 836  p^{26} + 600  p^{25} + 516  p^{24} + 284  p^{23} + 176  p^{22} + 40  p^{21} - 36  p^{20} - 140  p^{19} - 188  p^{18} - 228  p^{17} - 308  p^{16} - 288  p^{15} - 284  p^{14} - 188  p^{13} - 224  p^{12} - 188  p^{11} - 164  p^{10} - 96  p^{9} - 132  p^{8} - 88  p^{7} - 80  p^{6} - 48  p^{5} - 56  p^{4} - 24  p^{3} - 24  p^{2} - 24  p - 24
    \end{dmath}
    \begin{dmath}
        S_{1,5}(p) =
        52  p^{72} + 104  p^{71} + 292  p^{70} + 344  p^{69} + 584  p^{68} + 668  p^{67} + 1004  p^{66} + 1048  p^{65} + 1536  p^{64} + 1592  p^{63} + 2168  p^{62} + 2184  p^{61} + 2828  p^{60} + 2752  p^{59} + 3388  p^{58} + 3240  p^{57} + 3984  p^{56} + 3836  p^{55} + 4504  p^{54} + 4324  p^{53} + 5000  p^{52} + 4676  p^{51} + 5260  p^{50} + 4984  p^{49} + 5496  p^{48} + 5064  p^{47} + 5628  p^{46} + 5284  p^{45} + 5620  p^{44} + 5168  p^{43} + 5452  p^{42} + 4880  p^{41} + 5140  p^{40} + 4612  p^{39} + 4796  p^{38} + 4232  p^{37} + 4308  p^{36} + 3720  p^{35} + 3820  p^{34} + 3196  p^{33} + 3124  p^{32} + 2572  p^{31} + 2544  p^{30} + 1984  p^{29} + 1960  p^{28} + 1476  p^{27} + 1268  p^{26} + 912  p^{25} + 748  p^{24} + 340  p^{23} + 184  p^{22} + 56  p^{21} - 132  p^{20} - 252  p^{19} - 380  p^{18} - 404  p^{17} - 548  p^{16} - 496  p^{15} - 532  p^{14} - 356  p^{13} - 424  p^{12} - 324  p^{11} - 316  p^{10} - 200  p^{9} - 268  p^{8} - 168  p^{7} - 128  p^{6} - 80  p^{5} - 120  p^{4} - 40  p^{3} - 40  p^{2} - 40  p - 40
    \end{dmath}
    \begin{dmath}
        S_{3,5}(p) =
        236  p^{72} + 472  p^{71} + 956  p^{70} + 1192  p^{69} + 1912  p^{68} + 2324  p^{67} + 3252  p^{66} + 3624  p^{65} + 4928  p^{64} + 5416  p^{63} + 6904  p^{62} + 7272  p^{61} + 8964  p^{60} + 9056  p^{59} + 10772  p^{58} + 10840  p^{57} + 12752  p^{56} + 12628  p^{55} + 14520  p^{54} + 14332  p^{53} + 16120  p^{52} + 15468  p^{51} + 17188  p^{50} + 16592  p^{49} + 18008  p^{48} + 17072  p^{47} + 18588  p^{46} + 17532  p^{45} + 18620  p^{44} + 17304  p^{43} + 18220  p^{42} + 16640  p^{41} + 17284  p^{40} + 15836  p^{39} + 16332  p^{38} + 14520  p^{37} + 14788  p^{36} + 13120  p^{35} + 13188  p^{34} + 11276  p^{33} + 11076  p^{32} + 9396  p^{31} + 9120  p^{30} + 7480  p^{29} + 7208  p^{28} + 5708  p^{27} + 5108  p^{26} + 3944  p^{25} + 3332  p^{24} + 2140  p^{23} + 1616  p^{22} + 1032  p^{21} + 508  p^{20} - 76  p^{19} - 396  p^{18} - 612  p^{17} - 964  p^{16} - 1008  p^{15} - 1036  p^{14} - 748  p^{13} - 832  p^{12} - 652  p^{11} - 628  p^{10} - 480  p^{9} - 564  p^{8} - 344  p^{7} - 304  p^{6} - 240  p^{5} - 280  p^{4} - 120  p^{3} - 120  p^{2} - 120  p - 120. 
    \end{dmath}
\end{dgroup*}

\subsection{\texorpdfstring{Numerics for Theorem \ref{thm:stats}}{Numerics for Theorem 1.4}}
\label{subsec:numerics}

When the exponents $\ell,m,n$ are pairwise coprime, the hypotheses of Theorem \ref{thm:rat_funcs_body} are satisfied by Lemma \ref{lem:intermediate}. Thus $\Prob_{\ell,m,n}(p)$ is given uniformly by a single rational function in $p$. Theorem \ref{thm:stats} gives
\[\frac{N_{\ell,m,n}^{\mathrm{loc}}(T)}{\#\{P \in \P^2(\Q) : \Ht(P) \leq T\}} \sim \prod_p \Prob_{\ell,m,n}(p).\]

To estimate this numerically, we make use of the crude bounds
\begin{align}\label{eq:crude_bound}
    \left(1 - \frac1{p^2}\right)^4 \leq 1 - \frac3{p^2} - \frac1{p^3} + \frac3{p^5} \leq \Prob_{\ell,m,n}(p) \leq 1 - \frac4{p^3} + \frac3{p^6} &\leq 1 - \frac1{p^3}. & (p > 3)
\end{align}
The inner bound of \eqref{eq:crude_bound} is obtained by \eqref{eq:rho} and the trivial inequalities $0 \leq \ProbAB, \ProbAC, \ProbBC \leq 1$. The leftmost inequality holds for  $p > 3$, and the rightmost inequality holds for all $p$. Thus for any $M>3$ and $M'>0$ we have
\[\zeta(2)^{-4} \prod_{p \leq M} \left(1 - \frac1{p^2}\right)^{-4} \Prob_{\ell,m,n}(p) \leq \prod_p \Prob_{\ell,m,n}(p) \leq \zeta(3)^{-1} \prod_{p \leq M'} \left(1 - \frac1{p^3}\right)^{-1}\Prob_{\ell,m,n}(p).\]

For a concrete example, consider $(\ell,m,n) = (2,3,5)$. Taking $M=M' = 10000$ and using the explicit formula for $\Prob_{2,3,5}(p)$ given in Table \ref{tab:data}, we have
\begin{equation}\label{eq:235_numerics}
    0.78233 \leq \prod_p \Prob_{2,3,5}(p) \leq 0.78237
\end{equation}
The same approach applies in other cases of interest, e.g.\ $0.77498 \leq \prod_p \Prob_{2,3,7}(p) \leq 0.77502$.

\subsubsection{Numerics over number fields}
\label{subsubsec:numericsNFs}

Over a number field $K$, the bounds \eqref{eq:crude_bound} hold with $p$ replaced by $\Np$, allowing us to deduce similar bounds for the product of local densities in terms of special values of the Dedekind zeta function $\zeta_K(s)$: for $M > 3$ and $M' > 0$ we have
\begin{align}           
    \label{eq:crudelowerNF}\prod_{\p} \ProbK_{\ell,m,n}(\p) 
    &\geq \zeta_K(2)^{-4} \prod_{\Np \leq M} \left(1 - \frac1{\Np^2}\right)^{-4}\ProbK_{\ell,m,n}(\p)\\
    \label{eq:crudeupperNF} \prod_\p \ProbK_{\ell,m,n}(\p) &\leq \zeta_K(3)^{-1}\prod_{N_{K}(\p)\leq M'} \left (1 - \frac{1}{N_{K}(\p)^{3}}\right )^{-1}\ProbK_{\ell,m,n}(\p).
\end{align}

In particular, we deduce bounds for the limiting behavior of $\prod_{\p} \ProbK_{\ell,m,n}(\p)$ as $K$ ranges over number fields of fixed degree, independent of $\ell,m,n$. An immediate consequence is that if we allow the degree $[K:\Q]$ to be large, $\prod_\p \ProbK_{\ell,m,n}(\p)$ can be made arbitrarily close to 0 or 1, justifying Remark \ref{rem:prodNloc_bigsmall}.

\begin{proposition}
\label{prop:liminf_limsup_ProbK}
    \expcond{Suppose $\ell,m,n$ are pairwise coprime} and fix a degree $d > 1$. As $K$ ranges over degree $d$ number fields, we have
    \begin{align*}
        \liminf_{[K:\Q] = d} \prod_\p \ProbK_{\ell,m,n}(\p) &\leq \zeta(3)^{-d} \text{ and} \\
        \limsup_{[K:\Q] = d} \prod_\p \ProbK_{\ell,m,n}(\p) &\geq \zeta(2d)^{-4}.
    \end{align*}
\end{proposition}

\begin{proof}
    This relies on the well known fact that for any finite set of primes $S$, there exists a number field $K/\Q$ of degree $d$ such that all $p \in S$ are inert (resp.\ totally split) in $K$. Such a field $K = \Q[x]/(f(x))$ may be constructed using Sunzi's Remainder Theorem to build a monic polynomial $f$ with integer coefficients whose residue modulo $p$ is irreducible for all $p \in S$.
    
    To construct $K$ so that all $p \in S$ are totally split, it suffices to show $f$ has $d$ distinct roots in $\F_p$ for all $p$. However, if $p < d$ this is not possible. In this case, we can take $f_p = \sum_{0 \leq i \leq d} c_i x^{d-i} \in \Z/p^{\frac{d(d-1)}{2}}\Z$ for coefficients with valuation $v_p(c_i) = \sum_{0 \leq j \leq i} j$. If $f \in \Z[x]$ is any polynomial with $f \equiv f_p \pmod{p^{\frac{d(d-1)}{2}}}$, the $p$-adic Newton polygon of $f$ (see e.g.\ \cite[Proposition II.6.3]{Neukirch}) reveals that $f$ has $d$ distinct roots in $\Q_p$, hence $p$ is totally split in $K$. 

    Fix $M > 3$ and let $S = \{p \leq M\}$ and $K$ be a degree $d$ number field for which all $p \in S$ are inert in $K$. By \eqref{eq:crudelowerNF} and this choice of $K$ we find
    \[\prod_\p \ProbK_{\ell,m,n}(\p) \geq \zeta_K(2)^{-4} \geq \zeta(2d)^{-4} \prod_{p > M}\prod_{\p \mid p}\left(1 - \frac1{\Np^2}\right)^{4}.\]
    Choosing $M$ large, we can make the rightmost factor arbitrarily close to 1, establishing the desired limit superior.

    Similarly, if $K$ is a degree $d$ number field for which all $p \in S$ are totally split in $K$, then by \eqref{eq:crudeupperNF} and our choice of $K$, we have
    \[\prod_\p \ProbK_{\ell,m,n}(\p) \leq \zeta_K(3)^{-1} \leq \zeta(3)^{-d} \prod_{p > M}\left(1 - \frac1{p^3}\right)^{-d}.\]
    Letting $M$ grow large gives the desired limit inferior.
\end{proof}

\vfill\pagebreak
\subsection{\texorpdfstring{Explicit density functions for selected $(\ell,m,n)$}{Explicit density functions for selected (l,m,n)}}
\label{sec:table}

\renewcommand{\thetable}{\theequation}

\begin{center}
\begin{table}[h!]
\caption{Density functions $\rho_{\ell,m,n}(p) = 1 - \frac{f_{\ell,m,n}(p)}{g_{\ell,m,n}(p)}$ for polynomials $f_{\ell,m,n}$ and $g_{\ell,m,n}$}
\label{tab:data}
\renewcommand{\arraystretch}{1.2}
\begin{tabularx}{\textwidth}{|c|X|p{5cm}|c|} 
    \hline
    $(\ell,m,n)$ & $f_{\ell,m,n}(p)$ & $g_{\ell,m,n}(p)$ & valid $p$\\
    \hline
    \multicolumn{4}{|c|}{Spherical}\\
    \hline
    $(2,2,2)$ & $3p$ & $2(p+1)^2$ & $p \neq 2$ \\ \hline
    \makecell[t]{$(2,2,n)$\\ $n \geq 3$ \\ odd} 
        & $p^{3n}( p^5 + p^4 + 6p^3 + 6p^2 + 5p + 1 ) + p^{2n} ( -2p^5 - 10p^4 - 6p^3 - 4p^2 + 4) + p^n ( p^6 + 3p^5 - 2p^4 - 6p^3 - 7p^2 - 9p ) + 2p^5 + 6p^4 + 6p^3 + 4p^2$ 
        & $2(p^{2} + p + 1)(p^{2} + 1)\allowbreak(p + 1)^{2}\allowbreak (p^n + 1) (p^n - 1) p^{n}$ 
        & $p \neq 2$ \\ \hline 
     \makecell[t]{$(2,2,n)$ \\ $n \geq 4$ \\ even} 
        & $p^{2n}(3p^5 + 4p^4 + 8p^3 + 6p^2 + 5p + 1) + p^n(-p^6 - 2p^5 - 3p^4 - 2p^3 + p^2 + 1) - 3p^5 - 5p^4 - 6p^3 - 4p^2 - 3p$
        & $2(p^2 + p + 1)(p^2 + 1) (p+1)^2(p^n + 1)(p^n - 1)$
        & $p \neq 2$ \\ \hline
    $(3,3,2)$ 
        & $2  p^{11} + 4  p^{10} + 14  p^{9} + 10  p^{8} + 14  p^{7} + 8  p^{6} + 4  p^{5} + 4  p^{4} - 8  p^{2} - 6  p - 6$ 
        & $3 (p^{3} + 1) \allowbreak (p^{2} + p + 1)^{3} \allowbreak (p^{2} + 1) p$ & $p \equiv 1 \cmod{3}$ \\
        & $4  p^{9} + 2  p^{8} + 4  p^{7} + 2  p^{6} + 2  p^{4} - 2  p^{2} - 2  p - 2$ 
        & $(p^{3} + 1)\allowbreak(p^{2} + p + 1)^{3}\allowbreak (p^{2} + 1) p$ 
        & $p \equiv 2 \cmod{3}$ \\ \hline
    $(2,3,4)$ 
        & $p^{21} + p^{20} + 11  p^{19} + 8  p^{18} + 18  p^{17} + 12  p^{16} + 18  p^{15} + 17  p^{14} + 14  p^{13} + 7  p^{12} + p^{11} - 3  p^{10} - 4  p^{9} - 10  p^{8} - 7  p^{7} - 12  p^{6} - 9  p^{5} - 4  p^{4} - 5  p^{3} - 2  p^{2} - 2  p - 2$ 
        & $2  (p^{7} + p^{6} + 2  p^{5} + 2  p^{4} + 2  p^{3} + 2  p^{2} + p + 1)\allowbreak(p^{4} + 1)\allowbreak(p^{2} + p + 1)\allowbreak(p^{2} + 1)\allowbreak(p + 1) p^{6}$ 
        & $p \neq 2$ \\ \hline
    $(2,3,5)$ 
        & $4  p^{21} - p^{20} + 9  p^{19} - 2  p^{18} + 11  p^{17} - 5  p^{16} + 10  p^{15} - 6  p^{14} + 6  p^{13} - 2  p^{12} - 5  p^{9} - p^{8} - 6  p^{7} - 2  p^{6} - 4  p^{5} - 2  p^{3} - 2  p^{2} - 2$ 
        & $(p^{9} + p^{8} + p^{7} + p^{6} + p^{5} + p^{4} + p^{3} + p^{2} + p + 1)\allowbreak(p^{2} + p + 1)^{2}\allowbreak (p^{2} - p + 1)\allowbreak(p^{2} + 1) p^{7}$ 
        & all \\ \hline     
    \multicolumn{4}{|c|}{Elliptic}\\ \hline
    $(3,3,3)$ 
        & $2p$ & $p^2 + p + 1$ & $p \equiv 1 \cmod{3}$ \\
        & $6  p^{3}$ & $(p^2 + p + 1)^3$ & $p \equiv 2 \cmod{3}$ \\ \hline
    $(4,4,2)$ 
        & $7  p^{11} + 12  p^{10} + 27  p^{9} + 24  p^{8} + 34  p^{7} + 20  p^{6} + 30  p^{5} + 22  p^{4} + 23  p^{3} + 12  p^{2} + 7  p$ 
        & $4  (p^{7} + p^{6} + p^{5} + p^{4} + p^{3} + p^{2} + p + 1)\allowbreak (p^{3} + p^{2} + p + 1) \allowbreak (p^{2} + p + 1) $ 
        & \makecell[t]{$p \equiv 1 \cmod{4}$\\ $p \neq 5$}\footnotemark\\ 
        & $3  p^{11} + 5  p^{10} + 13  p^{9} + 12  p^{8} + 16  p^{7} + 9  p^{6} + 14  p^{5} + 11  p^{4} + 11  p^{3} + 5  p^{2} + 3  p$ 
        & $2(p^{7} + p^{6} + p^{5} + p^{4} + p^{3} + p^{2} + p + 1)\allowbreak (p^{3} + p^{2} + p + 1) \allowbreak (p^{2} + p + 1)$ 
        & $p \equiv 3 \cmod{4}$\\ \hline
    $(2,3,6)$
        & $7  p^{18} + 11  p^{17} + 30  p^{16} + 25  p^{15} + 38  p^{14} + 37  p^{13} + 47  p^{12} + 28  p^{11} + 40  p^{10} + 18  p^{9} + 33  p^{8} + 36  p^{7} + 23  p^{6} + 8  p^{5} + 12  p^{4} + 4  p^{3} - 9  p - 6$
        & $6 (p^{6} + 1)\allowbreak(p^{5} + p^{4} + p^{3} + p^{2} + p + 1)^{2}\allowbreak(p^{2} + p + 1) p$
        & $p \equiv 1 \cmod{3}$ \\
        & $p^{18} + p^{17} + 10  p^{16} + 7  p^{15} + 10  p^{14} + 11  p^{13} + 13  p^{12} + 8  p^{11} + 12  p^{10} + 2  p^{9} + 7  p^{8} + 12  p^{7} + 5  p^{6} + 4  p^{4} - 3  p - 2$
        & $2 (p^{6} + 1)\allowbreak(p^{5} + p^{4} + p^{3} + p^{2} + p + 1)^{2}\allowbreak(p^{2} + p + 1) p$
        & \makecell[t]{$p \equiv 2 \cmod 3$\\ $p \neq 2$}\\ \hline
    \multicolumn{4}{|c|}{Hyperbolic}\\ \hline
    $(4,4,4)$ 
        & $3(3p^{4} + 3p^{3} + 5p^{2} + 3p + 3)p$ & $4(p^{2}+1)^{2}(p+1)^{2}$ & $p \equiv 1 \cmod{4}$ \\
        & $3(p^{4} + p^{3} + 3p^{2} + p + 1)p$ & $2(p^{2} + 1)^{2}(p + 1)^{2}$ & $p \equiv 3 \cmod{4}$ \\
    & & & ($p > 31$)\\ \hline
    $(5,5,5)$ 
        & $6(2p^{4} + 2p^{3} + p^{2} + 2p + 2)(p^{2} + 1)p$ & $5(p^{4} + p^{3} + p^{2} + p + 1)^{2}$ & $p \equiv 1 \cmod{5}$ \\
        & $6(p^{2} + 1)p^{3}$ & $(p^{4} + p^{3} + p^{2} + p + 1)^{2}$ & $p \not\equiv 0,1 \cmod{5}$\\
    & & & ($p > 131$)\\ \hline
    $(2,3,7)$ 
        & $4  p^{30} - p^{29} + 9  p^{28} - p^{27} + 12  p^{26} - p^{25} + 10  p^{24} - 4  p^{23} + 9  p^{22} - 7  p^{21} + 8  p^{20} - 5  p^{19} + 6  p^{18} - 2  p^{17} + 2  p^{16} - 3  p^{15} - p^{14} - 7  p^{13} - p^{12} - 9  p^{11} - p^{10} - 4  p^{9} - 2  p^{8} - 2  p^{7} - p^{6} - 3  p^{5} - 3  p^{3} - p^{2} - 1$ 
        & $(p^{8} + 2  p^{7} + 3  p^{6} + 3  p^{5} + 3  p^{4} + 3  p^{3} + 3  p^{2} + 2  p + 1)\allowbreak(p^{7} + 1)\allowbreak(p^{2} + p + 1)\allowbreak(p^{2} - p + 1)\allowbreak(p^{2} + 1)p^{12}$ 
        & all \\ \hline 
\end{tabularx}
\end{table}
\end{center}

\footnotetext{When $(\ell,m,n) = (4,4,2)$ and $p=5$, $\Probvals{=0}{=0}{=0}^{(x)} = \Probvals{=0}{=0}{=0}^{(y)} = \Probvals{=0}{=0}{=0}^{(z)} = 7/8$.}

\vfill \pagebreak
\appendix
\usetikzlibrary{calc,intersections}

\definecolor{acolor}{rgb}{1.0, 0.8, 0.0} 
\definecolor{bcolor}{RGB}{0,159,183} 
\definecolor{ccolor}{RGB}{254,74,73} 

\newcommand{\halfpoint}[4]{%

\begin{scope}
\fill[#2] (#1) circle (#4);
\end{scope}

\begin{scope}
\clip (#1) rectangle ++(#4,-#4);
\fill[#3] (#1) circle (#4);
\end{scope}
\begin{scope}
\clip (#1) rectangle ++(#4,#4);
\fill[#3] (#1) circle (#4);
\end{scope}

\draw[black] (#1) circle (#4);
}

\newsavebox\PPw
\begin{lrbox}{\PPw}
\begin{tikzpicture}

\def\L{5}

\coordinate (A) at (0.5,-1);
\coordinate (B) at (1,0);
\coordinate (C) at (0,0);


\draw[ultra thick, acolor, name path=L1]
($(A)+(-\L,0)$) -- ($(A)+(\L,0)$);

\draw[ultra thick, bcolor, name path=L2]
($(B)+(-60:\L)$) -- ($(B)+(120:\L)$);

\draw[ultra thick, ccolor, name path=L3]
($(C)+(-120:\L)$) -- ($(C)+(60:\L)$);

\path[name intersections={of=L1 and L2, by=I12}];
\path[name intersections={of=L1 and L3, by=I13}];
\path[name intersections={of=L2 and L3, by=I23}];

\halfpoint{I12}{ccolor}{ccolor}{6pt}
\halfpoint{I13}{bcolor}{bcolor}{6pt}
\halfpoint{I23}{acolor}{acolor}{6pt}
\end{tikzpicture}
\end{lrbox}

\newsavebox\fermatstack
\begin{lrbox}{\fermatstack}
\begin{tikzpicture}

\def\L{5}

\coordinate (A) at (0.5,-1);
\coordinate (B) at (1,0);
\coordinate (C) at (0,0);


\draw[ultra thick, acolor, name path=L1]
($(A)+(-\L,0)$) -- ($(A)+(\L,0)$);

\draw[ultra thick, bcolor, name path=L2]
($(B)+(-60:\L)$) -- ($(B)+(120:\L)$);

\draw[ultra thick, ccolor, name path=L3]
($(C)+(-120:\L)$) -- ($(C)+(60:\L)$);

\path[name intersections={of=L1 and L2, by=I12}];
\path[name intersections={of=L1 and L3, by=I13}];
\path[name intersections={of=L2 and L3, by=I23}];

\halfpoint{I12}{ccolor}{ccolor}{6pt}
\halfpoint{I13}{bcolor}{bcolor}{6pt}
\halfpoint{I23}{acolor}{acolor}{6pt}

\draw[thick, name path=curve]
(-4,3) .. controls (3,3) and (-4,-1) .. (4,-3);

\path[name intersections={of=curve and L1, by=P1}];
\path[name intersections={of=curve and L2, by=P2}];
\path[name intersections={of=curve and L3, by=P3}];

\halfpoint{P1}{ccolor}{bcolor}{6pt}
\halfpoint{P2}{ccolor}{acolor}{6pt}
\halfpoint{P3}{acolor}{bcolor}{6pt}
\end{tikzpicture}
\end{lrbox}

\section{The weighted multiplicative action on generalized Fermat equations\\(by Santiago Arango-Pi\~neros, Christopher Keyes and Andrew Kobin)}
\label{sec:appendix}

Let us briefly reintroduce some notation, slightly generalizing the setup of \S \ref{subsec:GFEs}.

\begin{itemize}
    \item $\Rng$ is a Dedekind domain with fraction field $\K$.
    \item $k$ is a field and $k^{\sep}$ its separable closure.
    \item $A,B,C \in \Rng$ with $ABC\neq0$.
    \item $\l , \m , \n \in \ZZ_{\geq 1}$. 
    We let $L \coloneqq \lcm(\ell,m,n)$.
    \item The \emph{weight vector} $\bfw$ corresponding to $(\ell,m,n)$ is given by
    \[\bfw = (w_0, w_1, w_\infty) \coloneqq \left(\frac{L}{\ell}, \frac{L}{m}, \frac{L}{n}\right) \in \Z_{\geq 1}^3.\]  
    Note that $\gcd(w_0, w_1, w_\infty) = 1$.
    \item $\G_m$ is the multiplicative group, viewed as a group scheme over $R$, and $\G_m(\bfw)$ is the scheme theoretic image of $\G_m$ under the morphism  
    \[\Gm \to \Gm^3, \quad t \mapsto (t^{w_0}, t^{w_1}, t^{w_{\infty}}).\]
    \item $\mu_n$ is the $n$-torsion subgroup scheme of $\G_m$, i.e.\ the group scheme of $n$-th roots of unity.
\end{itemize}

\subsection{The weighted action on \texorpdfstring{$\AA^3-\mathbf{0}$}{A3-0}}
\label{sec:weighted-action-affine-space}
The action of $\Gm^3$ on $\AA^3-\mathbf{0}$ by coordinatewise multiplication
induces an action of $\Gm(\bfw)$ on the punctured affine space
$\AA^3-\mathbf{0}$. Since $\gcd(\bfw) = 1$, the map $\Gm \to \Gm(\bfw)$ is an
isomorphism. We let $\Gm$ act on $\AA^3-\mathbf{0}$ through $\Gm(\bfw)$. We compute the stabilizers of this action.

\begin{lemma}
  \label{lemma:stabilizer-data}
  Let $H_0 = V(\x)$, $H_1 = V(\y)$, and $H_{\infty} = V(\z)$
  denote the divisors on $\AA^3-\mathbf{0}$ at which the coordinates vanish. Then
   \begin{align*}
    \Stab_{\Gm}(H_0) &= \mu_{\gcd(w_1,w_{\infty})}, &
    \Stab_{\Gm}(H_0\cap H_1) &= \mu_{w_{\infty}},\\
    \Stab_{\Gm}(H_1) &= \mu_{\gcd(w_0,w_{\infty})},&
    \Stab_{\Gm}(H_0\cap H_\infty) &= \mu_{w_1}, \\
    \Stab_{\Gm}(H_\infty) &= \mu_{\gcd(w_0,w_1)},&
    \Stab_{\Gm}(H_1\cap H_\infty) &= \mu_{w_0}.
  \end{align*}
  Moreover, $\Gm$ acts without stabilizers on the complement of $H_0 \cup H_1 \cup H_\infty$.
\end{lemma}
\begin{proof}
  Let $T$ be any $R$-scheme, and abbreviate $S \coloneqq \Stab_{\Gm}(H_0)$.
  By definition,
  $$
  S(T) = \left\{t \in \OO_T(T)^{\times} : t \cdot P = P \text{ for every } P \in H_0(T)
  \right\}. 
  $$
  In particular, choosing any $t\in S(T)$ and $P = (0, 1 , 1)$ for
  $0,1 \in \OO_T(T)$, the equality $t \cdot P = P$ implies that $t^{w_1} = 1$ and
  $t^{w_{\infty}} = 1$. We conclude that
  $S(T) \subseteq \mu_{\gcd(w_1,w_{\infty})}(T)$. The reverse inclusion is clear. By Yoneda's
  lemma, we conclude the first equality in Lemma~\ref{lemma:stabilizer-data}. The computation of
  the other stabilizers follows similarly.

  For the final statement, note that if a point $P = (x,y,z) \in (\AA^3-\mathbf{0})(T)$ is stabilized by a non-trivial $t \in \Gm(T)$, we have that $t^{w_0}x = x, t^{w_1} y = y$, and $t^{w_{\infty}}z = z$. The fact that $\gcd(\bfw) = 1$ forces $xyz = 0$.
\end{proof}

We consider now the \emph{weighted projective stack}, given by the quotient of $\A^3 - \mathbf{0}$ by this weighted action,
\[\PPP(\bfw) \coloneqq [(\AA^3-\mathbf{0})/\Gm(\bfw)].\] 
Since the weighted projective space $\PP(\bfw)$ is the quotient scheme $(\AA^{3}-\mathbf{0})/\Gm(\bfw)$, it is also the coarse moduli space of $\PPP(\bfw)$. Let
$\pi\colon \PPP(\bfw) \to \PP(\bfw)$ denote the coarse map. One can think of $\pi$ as a map of degree one that is ``ramified'' over the points in $\PP(\bfw)$ for which the action of $\Gm(\bfw)$ on $(\AA^3-\mathbf{0})$ has non-trivial stabilizers.

\begin{lemma}
\label{lem:ramification-data}
  Let $W \subset \PP(\bfw)$ be the complement of the union of the three lines $H_0 = V(\x)$, $H_1 = V(\y)$, $H_\infty = V(\z)$. Then, the coarse map restricted to $\pi^{-1}(W)$ is an isomorphism. Moreover, a geometric point $Q\colon \Spec\bar{k} \to \PPP(\bfw)$ with $\pi(Q) \in H_0(\bar k) \cup H_1(\bar k) \cup H_\infty(\bar k)$ has stabilizer group isomorphic to $\mu_{m(Q)}$, where the positive integer $m(Q)$ is determined by the location of $\pi(Q)$, as in Lemma~\ref{lemma:stabilizer-data}.
\end{lemma}

\begin{proof}
  Let $U\subset \AA^3-\mathbf{0}$ be the complement of the three lines. The map $U\to W$ induced by the projection $(\AA^3 -\mathbf{0}) \to \PP(\bfw)$ is a $\Gm(\bfw)$-torsor, and so $[U/\Gm(\bfw)] \cong W$. On the other hand, $\pi^{-1}(W) = W \times_{\PP(\bfw)}\PPP(\bfw) \cong [U/\Gm(\bfw)]$. This proves the first statement.

  For the second statement, recall that $Q\colon \Spec \bar k \to \PPP(\bfw)$ is the data of a $\Gm(\bfw)_{\bar k}$-torsor $P \to \Spec \bar k$, together with a $\Gm(\bfw)$-equivariant map $\phi\colon P \to (\AA^3-\mathbf{0})$. Since $\bar k$ is algebraically closed, $P \to \Spec \bar k$ is isomorphic to the trivial torsor $\Gm(\bfw)_{\bar k} \to \Spec \bar k$. Since a $\Gm(\bfw)$-equivariant map $\Gm(\bfw)_{\bar k} \to (\AA^3-\mathbf{0})$ is determined by a point $q \in (\AA^3-\mathbf{0})(\bar k)$, the automorphism group of the diagram
  \[
    \begin{tikzcd}
        \Gm(\bfw)_{\bar k} \arrow[r, "\phi"] \arrow[d] & (\AA^3-\mathbf{0}) \\
        \Spec \bar k &
    \end{tikzcd}
  \]
  coincides with the subgroup of $\Gm(\bfw)(\bar k)$ stabilizing $q$. 
\end{proof}

\subsection{The weighted multiplicative action on the punctured cone}
\label{sec:weighted-action-cone}

Let $\RR \coloneqq R[x,y,z]/(\gfe)$ with the grading given by $\bfw$ and let $\RR_+$ denote its irrelevant ideal. Take $\UU \subset \A^3 - \mathbf{0}$ to be
\[\UU \coloneq \Spec \RR - V(\RR_+).\]
The action of $\Gm(\bfw)$ on $\AA^3 - \mathbf{0}$ considered in
Section~\ref{sec:weighted-action-affine-space} descends to an action on $\UU$. Indeed,
\begin{equation*}
    A(t^{w_0}x)^\l + B(t^{w_1}y)^\m + C(t^{w_\infty}z)^\n = \lambda(\gfe),
\end{equation*}
for $\lambda = t^{\l w_0} = t^{\m w_1} = t^{\n w_\infty}$. Therefore, the quotient stack $\X \coloneqq [\UU/\Gm(\bfw)]$ is a closed substack of $\PPP(\bfw)$, whose coarse moduli space is precisely the quotient scheme $X \coloneqq \UU/\Gm(\bfw) = \Proj \RR$. Let $\pi\colon \X \to X$ denote the coarse map.

\begin{lemma}
\label{lem:stackyness-fermat-stack}
  Let $W \subset \PP(\bfw)$ be the complement of the union of the three lines $H_0 = V(\x)$, $H_1 = V(\y)$, $H_\infty = V(\z)$. Then, the coarse map restricted to $\pi^{-1}(W\cap X)$ is
  representable by an isomorphism. Moreover, a geometric point $Q\colon \Spec \bar k \to \X$ with $\pi(Q) \in H_0(\bar k) \cup H_1(\bar k) \cup H_\infty(\bar k)$ satisfies $\pi(Q) \not\in \{[0:0:1], [0:1:0], [1:0:0]\}$ and has stabilizer group isomorphic to $\mu_{m(Q)}$, where 
  \[m(Q) = \begin{cases}
      \gcd(w_1,w_\infty), & \text{ if } \pi(Q) \in H_0(\bar k), \\
      \gcd(w_0,w_\infty), & \text{ if } \pi(Q) \in H_1(\bar k), \\
      \gcd(w_0,w_1), & \text{ if } \pi(Q) \in H_\infty(\bar k).
  \end{cases}\]
\end{lemma}
\begin{proof}
    The proof is analogous to that of Lemma~\ref{lem:ramification-data}. The point $\pi(Q)=[x:y:z]$ satisfies the equation $Ax^\l + By^\m + Cz^\n = 0$, and thus cannot be in the intersection of two lines.
\end{proof}

\begin{figure}[htbp]
  \centering
  \begin{subfigure}[b]{0.45\textwidth}
    \centering
    \resizebox{6cm}{!}{
    \begin{tikzpicture}
        \node at (0,0) {\usebox\PPw};
    \end{tikzpicture}
    }
    \caption{A geometric fiber of $\PPP(\bfw)$.}
    \label{fig:geometric-fiber-WPP}
  \end{subfigure}
  \hfill 
  \begin{subfigure}[b]{0.45\textwidth}
    \centering
    \resizebox{6cm}{!}{
    \begin{tikzpicture}
        \node at (0,0) {\usebox\fermatstack};
    \end{tikzpicture}
    }
    \caption{A geometric fiber of $\X \subset \PPP(\bfw)$.}
    \label{fig:geometric-fiber-fermat-stack}
  \end{subfigure}
\end{figure}

\subsection{Genus formulas}
\label{sec:genus-formulas}

Let $\X$ and $\Y$ be stacky curves over a field $\k$. Following \cite{vzb}, recall that the Euler characteristic $\chi(\X)$ of $\X$ is defined to be $-\deg \canonicaldiv_\X$, where $\canonicaldiv_\X$ is a canonical divisor on $\X$ \cite[Definition 5.5.8]{vzb}. 
We use $\ptcrs{\ptX}$ to denote the image of a closed point $P$ of $\X$ in its coarse space; if $\f \colon \Y \to \X$ is a representable surjective $k$-morphism of stacks, we use $\ptcrs{\f}\colon Y \to X$ to denote the induced morphism on coarse spaces. If $k' \supset \k$ is a field extension, and $\xi_P\colon \Spec k' \to \X$ is an $k'$-point in the equivalence class of $P$, we use $\Stab(\ptX)$ to denote the stabilizer $k'$-group scheme at $P$. Recall that when $P$ is a {\it tame} point, i.e.~$\operatorname{char}k$ does not divide the order of $\Stab(\ptX)$, one defines $\deg(P) \coloneqq \deg(|P|)/\#\Stab(P)(\k^{\sep})$ (this coincides with the degree of the residue gerbe at $P$ in the sense of \cite[Definition 5.1.3, Remark 5.2.3]{vzb}). 
We have the commutative diagram

\begin{equation}
    \label{diagram:morphisms-between-stacky-curves}
    \begin{tikzcd}
        Q \arrow[d, mapsto] & \Y \arrow[rr, "f"] \arrow[d] & & \X \arrow[d] & f(Q) \arrow[d, mapsto]  \\
        \ptcrs{Q} & Y \arrow[rr, "\ptcrs{f}"'] & & X & \ptcrs{f(Q)} 
    \end{tikzcd}
\end{equation}
For every tame point $Q$ in $\Y$, define
\begin{equation}
    e_f(\ptY) \coloneqq \frac{\# \Stab(\ptY) \cdot e_{\ptcrs{f}}(\ptcrs{\ptY})}{\# \Stab(f(\ptY))}.
\end{equation}
When $\X$ and $\Y$ are tame stacky curves, i.e.~all points on each curve are tame, we have the following version of a Riemann--Hurwitz formula. 

\begin{proposition}
\label{prop:relativeRH}
    Let $\f \colon \Y \to \X$ be a morphism of tame stacky curves over $\k$, such that the induced map on coarse spaces $\ptcrs{f}\colon Y \to X$ is finite and separable. Then
    \begin{equation}\label{eq:relativeRH}
    \chi(\Y) = \deg \ptcrs{f} \cdot \chi(\X) - \sum_{\ptY} \left(e_f(\ptY) - 1\right)\deg \ptY,
    \end{equation}
    where the sum is ranging over closed points $Q$ in $\Y$.
\end{proposition}

\begin{proof}
    The statement follows from combining \cite[Proposition 5.5.6, see also (5.5.10)]{vzb} for $\X$ and $\Y$ together with the usual Riemann--Hurwitz formula for $\ptcrs{\f}$, as follows.
    \begin{align*}
        \chi(\Y) &= \chi(Y) - \sum_\ptY \left( \# \Stab(\ptY) - 1 \right) \deg \ptY\\
        \chi(Y) &= \deg \ptcrs{f} \cdot \chi(X) - \sum_{\ptcrs{\ptY}} (e_{\ptcrs{f}}(\ptcrs{\ptY}) - 1)\deg \ptcrs{\ptY}\\
        \chi(\X) &= \chi(X) - \sum_\ptX \left( \# \Stab(\ptX) - 1 \right) \deg \ptX
    \end{align*}
    Rearranging and recalling $\sum_{f(\ptY) = \ptX} e_{\ptcrs{f}}(\ptcrs{\ptY})\deg \ptcrs{\ptY} = \deg \ptcrs{f} \cdot \deg \ptcrs{\ptX}$, we have
    \begin{align*}
        \chi(\Y) &= \deg \ptcrs{f}\cdot \chi(\X) + \deg\ptcrs{f} \cdot \sum_{\ptX} \left(1 - \frac1{\#\Stab(\ptX)}\right) \deg \ptcrs{\ptX}\\
        &\hspace{1cm}- \sum_\ptY \left(\#\Stab(\ptY) - 1\right) \deg \ptY - \sum_{\ptcrs{\ptY}} \left(e_{\ptcrs{f}}(\ptcrs{\ptY}) - 1\right) \deg \ptcrs{\ptY}\\
        &= \deg \ptcrs{f}\cdot \chi(\X) + \sum_{\ptY} \left(1 - \frac1{\#\Stab(f(\ptY))}\right) e_{\ptcrs{f}}(\ptcrs{\ptY})\deg \ptcrs{\ptY}\\
        &\hspace{1cm}- \sum_\ptY \left(\#\Stab(\ptY) - 1\right) \deg \ptY - \sum_{\ptcrs{\ptY}} \left(e_{\ptcrs{f}}(\ptcrs{\ptY}) - 1\right) \deg \ptcrs{\ptY}\\
        &= \deg \ptcrs{f}\cdot \chi(\X) - \sum_{\ptY} \left( \frac{\# \Stab(\ptY) \cdot e_{\ptcrs{f}}(\ptcrs{\ptY})}{\#\Stab(f(\ptY))} - 1 \right) \deg \ptY.
    \end{align*}
\end{proof}

Using Proposition~\ref{prop:relativeRH}, we can deduce the Euler characteristic of the stacky curve attached to any generalized Fermat equation, as well as that of its coarse moduli space. The formula for the coarse space appears in the literature, e.g.~in \cite[Corollary 3.5]{OW_1972}.

\begin{corollary}
\label{cor:GFE-euler-char}
    Fix positive integers $\ex,\ey$ and $\ez$ coprime to the characteristic of $K$. Set $\L = \lcm(\ex,\ey,\ez)$ and let $\X \colon \gfe = 0 \subset \PPP(\L/\ex,\L/\ey,\L/\ez)$. Then we have
    \begin{equation}\label{eq:GFE-euler-char}
    \chi(\X) = -\frac{\ex\ey\ez}{\L} + \frac{\gcd(\ex,\ey)}{\gcd(\L/\ex,\L/\ey)} + \frac{\gcd(\ex,\ez)}{\gcd(\L/\ex,\L/\ez)} + \frac{\gcd(\ey,\ez)}{\gcd(\L/\ey,\L/\ez)}.
    \end{equation}
    Denoting the coarse space of $\X$ by $\Xcrs$, we have
    \begin{equation}\label{eq:GFE-coarse-euler-char}
    \chi(X) = -\frac{\ex\ey\ez}{\L} + \gcd(\ex,\ey) + \gcd(\ex,\ez) + \gcd(\ey,\ez).
    \end{equation}
\end{corollary}

\begin{proof}
    The formulae \eqref{eq:GFE-euler-char} and \eqref{eq:GFE-coarse-euler-char} are equivalent by an application of Proposition \ref{prop:relativeRH} for the coarse moduli map $\X \to X$.
    We now derive the former.

    Let $Y\subset\PP^2$ be the plane curve defined by $Ax^{L} + By^{L} + Cz^{L} = 0$ and consider the morphism of stacky curves 
    \[\f \colon Y \longrightarrow \X\]
    induced by $(x,y,z) \mapsto (x^{L/\ex},y^{L/\ey},z^{L/\ez})$ on the affine cones. Then $\f$ satisfies the conditions of Proposition~\ref{prop:relativeRH} with $\chi(Y) = 3L - L^2$, $\deg \ptcrs{\f} = L^3/\ex\ey\ez$.
    Any geometric point $\ptY \in Y(K^{\sep})$ has degree one and no stabilizers, because $Y$ is a scheme.
    The order of the stabilizer $\Stab(f(\ptY))$ is given by Lemma \ref{lemma:stabilizer-data}. Using the fact that for any point $\ptX \in \X(K^{\sep})$ we have $\sum_{\f(\ptY) = \ptX} e_{\ptcrs{\ptY}} = \deg \ptcrs{\f} = L^3/\ex\ey\ez$, we apply \eqref{eq:relativeRH} to compute 
    \begin{align}
        \nonumber 3L - L^2 = \frac{L^3}{\ex\ey\ez} &\chi(\X) + L\left(\frac{L^2\gcd(\ey,\ez)}{\ex\ey\ez \lcm(L/\ey,L/\ez)} - 1 \right) \\
        &+ L\left(\frac{L^2\gcd(\ex,\ez)}{\ex\ey\ez \lcm(L/\ex,L/\ez)} - 1 \right) + L\left(\frac{L^2\gcd(\ex,\ey)}{\ex\ey\ez \lcm(L/\ex,L/\ey)} - 1 \right).
    \end{align}
    Rearranging, we obtain \eqref{eq:GFE-euler-char}.
\end{proof}

\begin{remark}
    In the proof of Corollary~\ref{cor:GFE-euler-char} above, we compute the local contributions to $\chi(\X)$ using the explicit group theory in Lemma~\ref{lemma:stabilizer-data}. Alternatively, one can compute the ramification indices in the cover of coarse spaces $Y\to X$ and obtain $\chi(X)$ first, then use \cite[Proposition 5.5.6]{vzb} in the other direction to deduce $\chi(\X)$. 
\end{remark}

Before presenting some applications of Corollary \ref{cor:GFE-euler-char} to specific families of generalized Fermat equations, we exhibit Proposition \ref{prop:relativeRH} for morphism of stacky curves which is not representable.

\begin{example}
    Assume that $k$ is algebraically closed of characteristic zero. Consider the action of the $k$-algebraic group $\mu_n$ on the projective line $\PP^1_k$ given on points by \[\zeta \cdot [s:t] \coloneqq [\zeta s : t].\]
    This action has two fixed points: the point $P_0 \coloneqq [0:1]$ and $P_\infty = [1:0]$, with $\Stab_{\mu_n}(P_0) = \Stab_{\mu_n}(P_\infty) = \mu_n(k)$. The stack quotient $[\PP^1/\mu_n]$ is therefore a stacky curve of Euler characteristic $\frac{2}{n}$.

    Now, let $\varphi\colon \mu_4 \to \mu_2$ be the $k$-homomorphism given on points by $\zeta \mapsto \zeta^2$. Importantly, $\ker\varphi$ is not trivial. Let $|f|\colon \PP^1 \to \PP^1$ be the squaring map $[s:t] \mapsto [s^2:t^2]$. Note that $|f|$ is $\varphi$-equivariant:
    \[|f|(\zeta \cdot P) = [s^2\zeta^2:t^2] = \zeta^2 \cdot [s^2:t^2] = \varphi(\zeta) \cdot |f|(P).\]
    Therefore, we get an induced morphism $f\colon [\PP^1/\mu_4] \to [\PP^1/\mu_2]$. This map restricts at the origin to the morphism $B\mu_4 \to B\mu_2$, which is \emph{not representable:} the pullback via the map $\Spec k \to B\mu_2$ given by the trivial $\mu_2$-torsor coincides with $[\mu_2/\mu_4]$ (here $\mu_4$ acts on $\mu_2$ via $\varphi$). Since $\ker\varphi$ is not trivial, $[\mu_2/\mu_4]$ is not an algebraic space. This implies that $f$ is not representable.

    Note that 
    \begin{align*}
        \chi(\Y) &= \chi([\PP^1/\mu_4]) = \tfrac12 \\
        \chi(\X) &= \chi([\PP^1/\mu_2]) = 1 \\
        \deg \ptcrs{f} &= 2 \\
        e_f(Q) &=
        \begin{cases}
            1, & \text{ if } Q \notin \left\{ [P_0/\mu_4], [P_\infty/\mu_4]\right\},\\
            2, & \text{ if } Q \in \left\{ [P_0/\mu_4], [P_\infty/\mu_4]\right\},
        \end{cases} \\
        \deg Q &= \frac{\deg|Q|}{\#\Stab(Q)} =
        \begin{cases}
         1, & \text{ if } Q \notin \left\{ [P_0/\mu_4], [P_\infty/\mu_4]\right\},\\
            \tfrac14, & \text{ if } Q \in \left\{ [P_0/\mu_4], [P_\infty/\mu_4]\right\}.
        \end{cases}
    \end{align*}
    This is consistent with \eqref{eq:relativeRH}. 
\end{example}

\subsection{Examples}
\label{sec:appendix-examples}

\begin{example}\label{ex:pairwisecoprime}
    Suppose that the exponents are pairwise coprime: \[\gcd(\ex,\ey) = \gcd(\ex,\ez) = \gcd(\ey,\ez) = 1.\] In this case, we have $L = \lcm(\ex,\ey,\ez) = \ex\ey\ez$ and 
    \[
        w_0 = \ey\ez,\quad
        w_1 = \ex\ez,\quad
        w_\infty = \ex\ey.        
    \]
    The stacky curve $\X$ has exactly three stacky points, one each with $x=0$, $y=0$, and $z=0$, with stabilizers $\mu_\ex$, $\mu_\ey$, and $\mu_\ez$, respectively.
    
    By \eqref{eq:GFE-euler-char} we have 
    \[\chi(\X) = -1 + \frac1\ex + \frac1\ey + \frac1\ez.\]
    The coarse space is a genus zero curve by \eqref{eq:GFE-coarse-euler-char}.
\end{example}

\begin{example}\label{ex:pairwisecoprimeweights}
    Suppose that the weights are pairwise coprime: 
    \[\gcd(L/\ex,L/\ey) = \gcd(L/\ex,L/\ez) = \gcd(L/\ey,L/\ez) = 1.\] 
    Then Corollary~\ref{cor:GFE-euler-char} shows that the coarse moduli map $\pi \colon \X\to X$ is an isomorphism; in particular, $\X\subset\PPP(L/\ex,L/\ey,L/\ez)$ is a scheme in this case. This is precisely the case that the weighted projective space $\PP(L/\ex,L/\ey,L/\ez)$ is \textit{well-formed}; see e.g.\ \cite[\S 5]{Iano-Fletcher_2000}.
\end{example}

\begin{example}\label{ex:coarseplanecurve}
    In \cite[Example 5.3.7]{vzb} and \cite[Lemmas 3.2.1.e and 3.2.2.b.ii]{SantiThesis}, it is claimed that for a generalized Fermat stacky curve $\X$, the coarse space is isomorphic to the plane curve $X_{g} \colon Ax^{g} + By^{g} + Cz^{g} = 0$, where $g = \gcd(\ell,m,n)$. This conclusion is incorrect in general. To see why, one can use \eqref{eq:GFE-coarse-euler-char} to compute the Euler characteristic of the coarse space of $\X$ and observe that it may disagree with that of a degree $g$ plane curve. For instance, $(\ex,\ey,\ez) = (6,10,15)$ provides a concrete counterexample (see Example \eqref{ex:pairwisecoprimeweights} above). Note that upon replacing $X_{g}$ with the correct coarse space, the rest of the statements in \cite[Lemmas 3.2.1.e and 3.2.2.b]{SantiThesis} hold. 
    
    However, under the hypothesis that $\ex' = \ex/g,\ey' = \ey/g$ and $\ez' = \ez/g$ are coprime, the coarse space of $\X$ {\it is} isomorphic to $X_{g}$, as we will show below.
    Write $\ex = g\ex',\ \ey = g\ey',\ \ez = g\ez'$ with $g = \gcd(\ex,\ey,\ez)$. Let $N \coloneqq L/g$, and consider the $N$-Veronese subring $\RR^{(N)}$ of $\RR$. Note that $\deg(\x^{\ex'}) = \deg(\y^{\ey'}) = \deg(\z^{\ez'}) = N$, and consider the injective graded homomorphism
    \begin{equation}\label{eq:graded-hom}
    R[X,Y,Z]/(AX^g + BY^g + CZ^g) \to \RR^{(N)},
    \end{equation}
    given by $X \mapsto \x^{\ex'}, Y \mapsto y^{\ey'}, Z \mapsto \z^{\ez'}$ (the grading on the left-hand-side is the homogeneous grading $\deg(X) = \deg(Y) = \deg(Z) = 1$). We claim that this map is an isomorphism if and only if $\gcd(\ex',\ey') = \gcd(\ex',
    \ez') = \gcd(\ey',\ez') = 1$.
    To see this, note that a monomial $\x^i\y^j\z^k \in \RR^{(N)}$ if and only if $N$ divides
    \[
    \deg(\x^i\y^j\z^k) = iw_0 + jw_1 + kw_\infty.
    \]
    Noting that $w_0/N = 1/\ex',\ w_1/N = 1/\ey',\ w_\infty/N = 1/\ez'$, we conclude that this is the case only if 
    \begin{equation}\label{eq:sum-or-fractions}
        \dfrac{i}{\ex'} + \dfrac{j}{\ey'} + \dfrac{k}{\ez'} \in \ZZ.
    \end{equation}
    On the other hand, $\x^i\y^j\z^k \in R[X,Y,Z]$ if and only if
    \begin{equation} \label{eq:divisibility-conditions}
        \ex'\mid i, \quad \ey' \mid j, \quad \ez' \mid k.
    \end{equation}
    Thus, the surjectivity of (\ref{eq:graded-hom}) is equivalent to the implication: (\ref{eq:sum-or-fractions}) $\Longrightarrow$ (\ref{eq:divisibility-conditions}). This implication is always true under the pairwise coprimality assumption. Since $\Proj \RR \cong \Proj \RR^{(N)},$ the isomorphism follows.
\end{example}

\begin{example}\label{ex:mmn}
    Suppose $\ex = \ey$ and let $\gmn = \gcd(\ey,\ez)$. Provided $\ey > 1$, these exponents are not pairwise coprime, and so long as $\ez \nmid \ey$, neither are the weights. Thus this family of examples sits somewhere in between Examples \ref{ex:pairwisecoprime} and \ref{ex:pairwisecoprimeweights}, and if further $\ey \nmid \ez$ then the coarse space fails to be neatly described by a plane curve as in Example \ref{ex:coarseplanecurve}. In this case, the formulae \eqref{eq:GFE-euler-char} and \eqref{eq:GFE-coarse-euler-char} specialize to 
    \begin{align*}
        \chi(\X) &= 2\gmn - \ey\gmn + \frac{\ey\gmn}{\ez}\\
        \chi(X) &= 2\gmn - \ey\gmn + \ey.  
    \end{align*}
\end{example}

\bibliographystyle{alpha}
\bibliography{refs}

\end{document}